\documentclass[12pt,oneside,reqno]{amsart}
\usepackage[all]{xy}
\usepackage{amsfonts,amsmath,oldgerm,amssymb,amscd}
\UseComputerModernTips
\numberwithin{equation}{section}
\usepackage[breaklinks]{hyperref}
\allowdisplaybreaks
\usepackage{comment}
\usepackage{enumerate}
\usepackage{caption} 
\usepackage{enumitem}   
\usepackage{comment}

\newtheorem{theorem}{Theorem}[section]
\newtheorem{proposition}[theorem]{Proposition}
\newtheorem{lemma}[theorem]{Lemma}
\newtheorem{coro}[theorem]{Corollary}
\newtheorem{conjecture}[theorem]{Conjecture}

\theoremstyle{definition}

\newtheorem{remark}{Remark}

\newcommand{\del}{\delta}

\newcommand{\Q}{\mbox{$\mathbb Q$}}
\newcommand{\R}{\mbox{$\mathbb R$}}     
\newcommand{\C}{\mbox{$\mathbb C$}}     

\begin{document}

	\title[ UNIFORM BOUNDS ON $S$-INTEGRAL POINTS IN BACKWARD ORBITS OF POWER MAPS]{ UNIFORM BOUNDS ON $S$-INTEGRAL POINTS IN BACKWARD ORBITS OF POWER MAPS} 

\author[R. Padhy]{R. Padhy}
\address{Rudranarayan Padhy, Department of Mathematics, National Institute of Technology Calicut, 
	Kozhikode-673 601, India.}
\email{rudranarayan\_p230169ma@nitc.ac.in; padhyrudranarayan1996@gmail.com}

\author[S. S. Rout]{S. S. Rout}
\address{Sudhansu Sekhar Rout, Department of Mathematics, National Institute of Technology Calicut, 
	Kozhikode-673 601, India.}
\email{sudhansu@nitc.ac.in; lbs.sudhansu@gmail.com}

\thanks{2020 Mathematics Subject Classification: Primary 37F10, Secondary 11G50, 11J86. \\
	Keywords: Backward orbit; equidistribution; integral points; preperiodic points; exceptional points; The Arakelov-Zhang pairing; linear forms in logarithms.\\
	Work supported by NBHM grant (Sanction Order No: 14053).}

\begin{abstract}
Let $K$ be a number field with algebraic closure $\overline{K}$ and let $S$ be a finite set of places of $K$ containing all the archimedean places. It is known from Silverman's result that a forward orbit of a rational map $\varphi$ contains finitely many $S$-integers in the number field K when $\varphi^2$ is not a polynomial. Sookdeo stated an analogous conjecture for the backward orbits of a rational map $\varphi$ using a general $S$-integrality notion based on the Galois conjugates of points.  He proved his conjecture for the power map  $\varphi(z) =z^d$ for $d \geq 2$ and consequently for Chebyshev maps (J. Number Theory 131 (2011), 1229-1239). In this paper, we establish uniform bounds on the number of $S$-integral points in the backward orbits of a fixed non-zero $\beta$ in $K$, relative to a non-preperiodic point $\alpha \in \mathbb{P}^1(\overline{K})$, under the power map $\varphi(z) =z^d $. 
\end{abstract}

\maketitle
\pagenumbering{arabic}
\pagestyle{headings}

\section{Introduction}
Let $K$ be a number field with algebraic closure $\overline{K}$, let $S$ be a finite set of places of $K$ containing all the archimedean places of $K$ and let $\alpha, \beta \in \overline{K}$.  We say that $\beta$ is \emph{$S$-integral relative to} $\alpha$ if no conjugate of $\beta$ meets any conjugate of $\alpha$ at primes lying outside of $S$. For a rational map $\varphi: \mathbb{P}^1 \to \mathbb{P}^1$ defined over a field $K$, we say that a point $x\in \mathbb{P}^{1}(\overline{K})$ is preperiodic if there exist distinct positive integers $m, n$ such that $\varphi^m(x) = \varphi^n(x)$.  The \emph{forward orbit} of $x\in \mathbb{P}^1(K)$ under $\varphi$ is defined as 
\begin{equation*}
	\mathcal{O}^+_\varphi(x)=\{x, \varphi(x),\varphi^2(x),...\}
\end{equation*}
and the \emph{backward orbit} of $x$, denoted by $\mathcal{O}^-_\varphi(x)$, is defined as  
\begin{equation*}
	\mathcal{O}^-_\varphi(x)=\bigcup_{n\geq 0} \varphi^{-n}(x)\subset \mathbb{P}^{1}(\overline{K}),  
\end{equation*} 
where $\varphi^{-n}(x)=\{y\in \mathbb{P}^1(\overline{K}):\varphi^n(y)=x\}$.
Furthermore, a point $x$ is said to be \emph{preperiodic} for $\varphi$ or $\varphi$-\emph{preperiodic} if its forward orbit $\mathcal{O}^+_\varphi(x)$ is finite. Similarly, $x$ is \emph{exceptional} if its backward orbit $\mathcal{O}^-_\varphi(x)$ is finite. We denote by $\mbox{PrePer}(\varphi, \overline{K})$ the set of all $\varphi$-preperiodic points in $\overline{K}$. 

There are strong similarities between $\mathcal{O}_\varphi^{-}(\beta)$ and $\mbox{PrePer}(\varphi, \overline{K})$. It has been conjectured that the set PrePer$(\varphi, \overline{K})$ contains only finitely many points that are $S$-integral relative to a non-preperiodic point $\alpha$ under $\varphi$.
In 1993, Silverman \cite{sil93} proved that if $\varphi^2(z)$ is not a polynomial, then $\mathcal{O}_\varphi^+(\beta)$ contains at most finitely many points in $\mathcal{O}_{K,S}$, the ring of $S$-integers in $K$. Since $\mathcal{O}_\varphi^-(\beta) \cap \mathbb{P}^1(L)$ is finite for any $\beta$ and for any finite extension $L$ of $K$ \cite[Corollary 2.2]{Sookdeo}, it is natural to ask when $\mathcal{O}_\varphi^-(\beta)$ contains finitely many points in $\mathcal{O}_{K,S}$.
Sookdeo made the following conjecture in \cite{Sookdeo}.
\begin{conjecture}[\cite{Sookdeo}, p.1230]\label{conj1} If $\alpha \in \mathbb{P}^1(K)$ is not $\varphi$-preperiodic, then for any $\beta \in \mathbb{P}^1(K)$, $\mathcal{O}_{\varphi}^-(\beta)$ contains at most finitely many points in $\mathbb{P}^1(\overline{K})$ which are $S$-integral relative to $\alpha$.
\end{conjecture}

 Note that, Conjecture~\ref{conj1} is true when $\varphi$ is a power map, i.e., $\varphi(z)=z^d$ with $d\ge 2$, by bounding the number of Galois orbits of the polynomial $z^n-\beta$ when $\beta$ is neither $0$ nor a root of unity, and using Siegel's theorem for integral points on $\mathbb{G}_m(K)$ (see \cite{Sookdeo}). By the functorial property of the relative $S$-integral, Conjecture \ref{conj1} is also true for Chebyshev map \cite{{Sookdeo}}. Recently, we have seen the effectiveness of quantitative equidistribution techniques in answering questions in unlikely intersections, particularly with a view toward uniform result \cite{demarco2020, demarco2022, yap}.
 
 The finiteness result in Conjecture \ref{conj1} for the map $\varphi(z)=z^d$ follows immediately since $\mathcal{O}^{-}_\varphi(\beta) \subset \{\gamma \in \overline{K}~|~\gamma^n= \beta, \text{for some } n \in \mathbb{Z}_{\geq 0}\}$.
In this paper, we provide quantitative bounds on the size of Galois orbits of integral points in backward orbits of power maps. For an algebraic number $x\in \overline{\Q}$ and a number field $K$, we let $G_K(x):=\mbox{Gal}(\overline{K}/K)\cdot x$ denote the $\mbox{Gal}(\overline{K}/K)$-orbit of $x$ and $|G_K(x)|$ denote the size of the set $G_K(x)$.  
\begin{theorem}\label{thm2}
Let $S$ be a finite set of places of $\mathbb{Q}$ including all the archimedean places and $K$ be a number field. Let $d\geq 2$ be an integer, $\varphi(z)=z^d$ be a rational map, $\alpha  \not \in \emph{PrePer}(\varphi, \overline{K})$ and $\beta$ be a fixed non-zero element in $K$. Then there exists a constant $C=C([K:\mathbb{Q}],S, d, \beta)$ such that for any such $\alpha$, if $\gamma \in \mathcal{O}^-_\varphi(\beta)$ is $S$-integral relative to $\alpha$ then $|G_{K}(\gamma)| <C$.
\end{theorem}
\begin{remark}
In Theorem \ref{thm2}, we allow $\alpha$ to vary over number fields $K$ with $[K:\Q] \leq D$ for some positive integer $D \geq 1$, which is why we let $S$ be a finite set of places of $\mathbb{Q}$ rather than $K$. 
\end{remark}
\begin{remark}
For fixed $S$ and $\beta$, our result is uniform in $\alpha$. If we vary $\beta$, Theorem \ref{thm2} is not true, which can be observed from the following example. Let $K=\mathbb{Q}$, and $S = \{\infty\}$.  Given $d>1$ and $\alpha:=2$, take $\beta = 2^{d^n}+1, \gamma= \beta^{1/d^n}$ for every $n=1,2, \dots$. Then $\gamma^{d^n}=2^{d^n}+1$ is $S$-integral with respect to $\alpha^{d^n}=2^{d^n}$, so necessarily $\gamma$ is $S$-integral with respect to $\alpha$. But the degree of $\gamma$ over $\mathbb{Q}$ tends to infinity with $n$. This is a contradiction to Theorem \ref{thm2}. 
\end{remark}
Next, we will provide an upper bound that grows exponentially with $[K:\Q]$. We let $S_{\mbox{fin}}$ be the subset of $S$ consisting of all non-archimedean places. 
\begin{theorem}\label{thm02} 
Let $K$ be a number field, $S$ be a finite set of places of $K$ including all the archimedean places. Let $d \geq 2$ be an integer, $\varphi(z)=z^d$ be a rational map, $\alpha  \not \in \emph{PrePer}(\varphi, \overline{K})$ and $\beta$ be a fixed non-zero element in $K$. Then there exists a constant $C_1=C_1(d, \beta) >0$ such that for any such $\alpha$,  the set 
$$\{\gamma \in \mathcal{O}^-_\varphi(\beta): |G_K(\gamma)|  >C_1 |S|^3 [K:\Q]^{8}, \gamma \;\text{is $S$-integral relative to $\alpha$}\}$$ is a union of at most $|S_{\emph{fin}}|\;\; \emph{Gal}(\overline{K}/K)$-orbits.
\end{theorem}

The plan of our paper is as follows. In Section \ref{sec-prelim}, we set up our notation and give a brief overview of Berkovich spaces, the Arakelov-Zhang pairing, and the quantitative equidistribution theorem. Section \ref{sec-finiteness} contains a proof of the finiteness of the $S$-integral points in backward orbits. Finally, in Section \ref{sec-proof}, we provide a proof of our main theorem along with some auxiliary results. We note that the main ideas of our work are coherent with \cite{yap}.
\section{Preliminaries}\label{sec-prelim}
\subsection{Heights}\label{height}
Let $K$ be an algebraic number field and denote its ring of integers by $\mathcal{O}_K$. Let $M_K$ represent the set of places of $K$, i.e. the set of normalized, inequivalent absolute values in $K$. For each $v \in M_K$, we write $K_v$ for the completion of $K$ at the place $v$ and $\overline{K}_v$ be the algebraic closure of $K_v$.
When $v$ is an infinite (archimedean) place, we define the absolute value on $\mathbb{Q}$ by
\begin{equation}\label{7abc}
	|x|_v := |x|^{[K_v:\mathbb{R}]/[K:\mathbb{Q}]} \quad \text{for } x \in \mathbb{Q}.
\end{equation}
If $v$ is a finite place lying above a rational prime $p$, we set
\begin{equation}\label{7abcd}
	|x|_v := |x|_p^{[K_v:\mathbb{Q}_p]/[K:\mathbb{Q}]} \quad \text{for } x \in \mathbb{Q}.
\end{equation}
For each $p \in M_{\mathbb{Q}}$, we fix a normalized absolute value $|\cdot|_p$ as follows: if $p = \infty$, then $|\cdot|_p$ is the usual absolute value on $\mathbb{Q}$; if $p$ is a finite prime, then $|\cdot|_p$ is the $p$-adic absolute value normalized so that $|p|_p = 1/p$. With this setup, for any $x \in K$ and place $v$ of $K$ lying above $p$, we have
\begin{equation}\label{eq8}
	|x|_v := |N_{K_v/\mathbb{Q}_p}(x)|_p^{1/[K:\mathbb{Q}]}.
\end{equation}
These absolute values satisfy the product formula: $\prod_{v \in M_K} |x|_v = 1$ for all $x \in K^\times$.
More generally, if $L/K$ is a finite extension and $x \in L$, then $\prod_{v \in M_K} \prod_{\sigma} |\sigma(x)|_v = 1$, where the inner product is taken over all $[L:K]$ field embeddings $\sigma: L \to \mathbb{C}_v$. 

Now consider a point $x = (x_1 : x_2) \in \mathbb{P}^1(\overline{K})$, where $x_1, x_2 \in L$ for some finite extension $L/K$. The \emph{absolute logarithmic height} of $x$ is defined as
\begin{equation}\label{eq9}
	h(x) = \sum_{w \in M_L} N_w\log \max\left( |x_1|_w, |x_2|_w \right) = \sum_{v \in M_K} \sum_{w|v} N_w \log \max\left( |x_1|_w, |x_2|_w \right),
\end{equation}
where $N_w = [L_w:\mathbb{Q}_w]/[L:\mathbb{Q}]$. 

Further, for algebraic numbers $\alpha_1, \dots, \alpha_r$, the height satisfies the inequality:
\begin{equation}\label{height_upper_bound}
	h(\alpha_1 + \cdots + \alpha_r) \leq h(\alpha_1) + \cdots + h(\alpha_r) + \log r.
\end{equation}
Let $\varphi: \mathbb{P}^1\to \mathbb{P}^1$ be a rational map of degree $d\geq 2$ defined over $K$ and a place $v \in M_K$, let $\mu_{\varphi, v}$ be the equilibrium measure of $\varphi$ for $v$. For an integer $n\geq 1$, we denote by $\varphi^n$ to mean the $n$-time composition of $\varphi$ with itself. The Call-Silverman canonical height function $h_{\varphi}: \mathbb{P}^1(\overline{K}) \to \mathbb{R}$ relative to $\varphi$ is defined by
\begin{equation}\label{eqcsheight}
	h_{\varphi}(x) := \lim_{n\to \infty} \frac{h(\varphi^n(x))}{d^n}
\end{equation}
for all $x\in \mathbb{P}^1(\overline{K})$. It is shown that the limit in \eqref{eqcsheight} exists (see \cite{CS}). The canonical height is uniquely characterized by the following properties: for all $x\in \mathbb{P}^1(\overline{K})$,
\begin{align*}
	&h_{\varphi}\left(\varphi(x)\right) = d h_{\varphi}(x)\quad \mbox{and}\quad |h(x) - h_{\varphi}(x)|_v < C_{\varphi}
\end{align*}
where $C_{\varphi}$ is an absolute constant depending on $\varphi$. One of the basic properties of the canonical height is that $h_{\varphi}(x)=0$ if and only if $x$ is preperiodic (see \cite[Theorem 3.22]{sil}).

\subsection{$S$-integrality}\label{integrality}
Let $S$ be a finite subset of the set of places $M_K$ of a number field $K$, including all archimedean places and $\mathcal{O}_{K,S}$ be the ring of $S$-integers is defined by $$ \mathcal{O}_{K,S} = \{ x \in K \mid |x|_v \le 1 \text{ for all } v \notin S \}.
$$ The $v$-adic chordal metric on $\mathbb{P}^1(\mathbb{C})$ is defined for any $v \in M_K$ by
\begin{equation}
	\delta_v(x, y) = \frac{|x_1 y_2 - y_1 x_2|_v}{\max\{|x_1|_v, |x_2|_v\} \cdot \max\{|y_1|_v, |y_2|_v\}},
\end{equation}
where $x = (x_1 : x_2)$ and $y = (y_1 : y_2)$ represent points in homogeneous coordinates on $\mathbb{P}^1$.

Since $\delta_v(x, y)$ always lies between $0$ and $1$, this allows us to interpret $\mathcal{O}_{K, S}$ as the set of points $\gamma \in K$ such that for all places $v \notin S$, the $v$-adic chordal distance between $(\gamma : 1)$ and the point at infinity is maximal. More precisely, for all $v \notin S$, the condition $|\gamma|_v \leq 1$ holds if and only if $\delta_v((\gamma : 1), \infty) = 1$.

Define the local height function $\lambda_{x, v}(y) = -\log \delta_v(x, y)$ for $x, y \in \mathbb{P}^1(\mathbb{C}_v)$ (see \cite[Chapter 3]{sil}). Let $\alpha$ and $\beta$ be two points in $\mathbb{P}^1(\overline{K})$. We say that $\beta$ is $S$-integral with respect to $\alpha$ if for all $v \notin S$ and for all embeddings $\sigma, \sigma' : \overline{K} \hookrightarrow \mathbb{C}_v$ over $K$, we have $\lambda_{\sigma'(\alpha), v}(\sigma(\beta)) = 0$. If we identify a point $x \in \overline{K}$ with $(x : 1) \in \mathbb{P}^1(\overline{K})$, then the condition for a point $y \in \overline{K}$ to be $S$-integral relative to $x$ can be written in terms of absolute values as follows: for every place $v \notin S$ and every pair of embeddings $\sigma, \sigma'\in \text{Gal}(\overline{K}/K)$,
\begin{equation*}
	\begin{cases}
		|\sigma(y) - \sigma'(x)|_v \geq 1 \quad &\text{if } |\sigma'(x)|_v \leq 1, \\
		|\sigma(y)|_v \leq 1 \quad &\text{if } |\sigma'(x)|_v > 1.
	\end{cases}
\end{equation*}
\subsection{The Berkovich affine and projective lines}\label{berko}
In this section, we will review the definitions of the Berkovich affine and projective lines. For more details on these objects (see \cite{baker2010, berko}). In this section, let $(K, |\cdot|)$ denote either the field $K =\C$ with its usual absolute value, or an arbitrary algebraically closed field $K$ which is complete with respect to a non-trivial, non-archimedean absolute value. The Berkovich affine line $\mathbb{A}_{\mathrm{Berk},v}^1$ is defined to be the set of multiplicative seminorms on the polynomial ring $K[T]$ in one variable. We use $[\cdot]_x$ to denote the seminorm corresponding to the point $x\in \mathbb{A}_{\mathrm{Berk},v}^1$. Observe that given an element $a\in K$, we have evaluation seminorm $[f(T)]_a =|f(a)|$. The map $a \mapsto [\cdot]_a$ defines a dense embedding $K \hookrightarrow \mathbb{A}_{\mathrm{Berk},v}^1$ and hence this identifies each point $a\in K$ with its corresponding seminorm $[\cdot]_a$ in $\mathbb{A}_{\mathrm{Berk},v}^1$. Any element $f(T)\in K[T]$ extends to a function on $\mathbb{A}_{\mathrm{Berk},v}^1 \to \mathbb{R}_{\geq 0}$ and the topology on $\mathbb{A}_{\mathrm{Berk},v}^1$ is given the weakest topology such that all functions of the form $x\mapsto [f(T)]_x$ are continuous. This makes $\mathbb{A}_{\mathrm{Berk},v}^1$ into a locally compact, Hausdorff and path connected topological space. The Berkovich projective line $\mathbb{P}_{\mathrm{Berk},v}^1 = \mathbb{A}_{\mathrm{Berk},v}^1 \cup \{\infty\}$ of $\mathbb{A}_{\mathrm{Berk},v}^1$. Also, the dense inclusion map $K \hookrightarrow \mathbb{A}_{\mathrm{Berk},v}^1$ extends to a dense inclusion map $\mathbb{P}(K) \hookrightarrow \mathbb{P}_{\mathrm{Berk},v}^1$ defined by $(a:1)\mapsto [\cdot]_a$ and $(1:0)\mapsto \infty$. Since $\mathbb{A}_{\mathrm{Berk},v}^1$ is locally compact, $\mathbb{P}_{\mathrm{Berk},v}^1$ is compact.

If $K$ is archimedean, that is $K= \C$, then by the Gelfand-Mazur theorem in functional analysis, the evaluation seminorms are the only seminorms. Thus, $\mathbb{A}_{\mathrm{Berk},v}^1= \C$ and $\mathbb{P}_{\mathrm{Berk},v}^1 = \mathbb{P}^1(\C)$. 

When $K$ is non-archimedean, the inclusion $\mathbb{P}^1(K) \hookrightarrow \mathbb{P}_{\mathrm{Berk},v}^1$ is not surjective. For example, each closed disk $D(a, r) = \{z\in K\mid |z-a|\leq r\}\; $ for $a\in K$ and $r\in |K|$ defines a point $\zeta_{a, r}\in \mathbb{P}_{\mathrm{Berk},v}^1$ corresponding to the sup-norm $[f(T)]_{\zeta_{a,r}} = \sup_{z\in D(a, r)}|f(z)|$. Note that each $a\in K$ can also be written as the point $\zeta_{a, 0}\in \mathbb{A}^1(\C) \hookrightarrow \mathbb{P}_{\mathrm{Berk},v}^1$ corresponding to a disk of radius zero. These are known as Type I points, or classical points. When $r$ is an element of value group, we say that $\zeta_{a, r}$ is a Type II point, otherwise $\zeta_{a, r}$ is a Type III point. The remaining points of $\mathbb{A}_{\mathrm{Berk},v}^1$ are called Type IV points, and they correspond to a nested intersection of disks $\cdots D_n \subseteq \cdots \subseteq D_1$ such that $\cap_{n=1}^{\infty} D_n = \emptyset$ but their radii do not go to zero. Moreover, if we fix a point $\zeta \in \mathbb{P}^1_{\mathrm{Berk}}(\mathbb{C}_v)$,
the set $[x,\zeta] \cap [y,\zeta] \cap [x,y]$ consists of a single point, which we denote by $x \vee_\zeta y$. We write $\vee$ for $\vee_\infty$ and $\vee_G$ for $\vee_{\zeta_G}$.

The Laplacian $\Delta$ on the Berkovich projective line $\mathbb{P}_{\mathrm{Berk},v}^1$ is an operator which assigns to a continuous function $f: \mathbb{P}_{\mathrm{Berk},v}^1 \to \R\cup \{\pm \infty\}$ a signed Borel measure $\Delta f$ on $\mathbb{P}_{\mathrm{Berk},v}^1$. For example, in the archimedean case $K = \C$,  assume that $f: \mathbb{P}_{\mathrm{Berk},v}^1 \to \R$ is twice continuously real-differentiable, we identify $\mathbb{P}_{\mathrm{Berk},v}^1= \mathbb{P}^1(\C) = \C\cup \{\infty\}$ by the affine coordinate $z= (z:1)\in \C$, where $\infty = (1:0)$, and let $z= x+iy$ for $x, y\in \R$. Then 
\[\Delta f(z) = -\frac{1}{2 \pi} \left(\frac{\partial^2}{\partial^2 x} + \frac{\partial^2}{\partial^2 y}\right)f(z)dx dy.\]
The Berkovich space allows one to develop  a suitable analogue of the Laplacian for non-archimedean places. It has been shown by Baker-Rumley \cite{baker2010}, Favre-Rivera-Litelier \cite{favre}, that the non-archimedean Berkovich projective line $\mathbb{P}_{\mathrm{Berk},v}^1$ carries an analytic structure, and in particular a Laplacian, which is very similar to its archimedean counterpart. For a detailed construction of Lapalcian on $\mathbb{P}_{\mathrm{Berk},v}^1$  (see \cite{baker2010}). Also, this measure valued Laplacian satisfy the self adjoint property:
\[\int f \Delta g = \int g \Delta f.\]

Let $K$ be a number field. For each $v\in M_K$, there is a distribution-valued Laplacian operator $\Delta$ on $\mathbb{P}_{\mathrm{Berk}, v}^1$. The function $\log^{+}|z|_v:=\log \max\{|z|_v,1\}$ on $\mathbb{P}^1(\C_v)$ extends naturally to a continuous real valued function $\mathbb{P}_{\mathrm{Berk}, v}^1\to \R\cup\{\infty\}$, and the Laplacian is normalized such that 
$$\Delta \log^{+}|z|_v = \delta_0-\delta_{\infty}$$ on $\mathbb{P}_{\mathrm{Berk},v}^1$, where $\delta_0= m_{\mathbb{S}^1}$ is the Lebesgue probability measure on the unit circle $\mathbb{S}^1$ when $v$ is archimedean, and $ \delta_0$ is a point mass at the Gauss point $\zeta_{0,1}$ of $\mathbb{P}_{\mathrm{Berk},v}^1$ when $v$ is non-archimedean.

\subsection{The Arakelov-Zhang pairing}
For two rational maps $\varphi$ and $\psi$ defined on $\mathbb{P}^1$ over a number field $K$, each having degree at least two, the Arakelov-Zhang pairing $\langle \varphi, \psi \rangle$ captures a deep relationship between their respective canonical height functions $h_{\varphi}$ and $h_{\psi}$. 

To define Arakelpv-Zhang pairing we recall some notation (see \cite{petsche2012}). Let $K$ be a number field and let $M_K$ denote the set of places of $K$. For each $v \in M_K$, let $K_v$ be the completion of $K$ at $v$, and let $\mathbb{C}_v$ be the completion of an algebraic closure of $K_v$. Let $\mathcal{L}$ be a line bundle on $\mathbb{P}^1$. For each place $v \in M_K$, the line bundle $\mathcal{L}$ extends to a line bundle $\mathcal{L}_v$ on the Berkovich projective line $\mathbb{P}^1_{\mathrm{Berk},v}$. A continuous metric $\|\cdot\|_v$ on $\mathcal{L}_v$ is a continuous function $\|\cdot\|_v : \mathcal{L}_v \to \mathbb{R}_{\ge 0}$ which induces a norm on each fiber
$\mathcal{L}_z$ as a $\mathbb{C}_v$-vector space. The metric $\|\cdot\|_v$ is said to be \emph{semi-positive} if for any section $s$, the function $\log |s(z)|$ on $\mathbb{P}^1_{\mathrm{Berk},v}$ is subharmonic. It is said to be \emph{integrable} if $\log |s(z)|:\mathbb{P}^1 \to \mathbb{R} \cup \{-\infty\}$ can be written as the difference of two subharmonic functions.      

For $\mathcal{L} = \mathcal{O}(1)$, there is a standard metric $\|\cdot\|_{\mathrm{st},v}$ defined by
$$\|s(z)\|_{\mathrm{st},v} = \frac{|s(z_1,z_2)|_v}{\max\{|z_1|_v, |z_2|_v\}},$$ where $s(z_1,z_2)$ is the linear homogeneous polynomial in $\mathbb{C}_v[z_1,z_2]$ representing the section $s$, and
$z = (z_1:z_2) \in \mathbb{P}^1(\mathbb{C}_v)$. An integrable adelic metric on $\mathcal{O}(1)$ is a family of metrics $(\|\cdot\|_v)_{v \in M_K}$ such that $\|\cdot\|_v$ is a continuous integrable metric on $\mathcal{O}(1)$ over $\mathbb{P}^1_{\mathrm{Berk},v}$ for each $v$, and $\|\cdot\|_v = \|\cdot\|_{\mathrm{st},v}$ for all but finitely many places $v$.

For a global section $s \in \Gamma(\mathbb{P}^1, \mathcal{O}(1))$ and a point $z \in \mathbb{P}^1(K) \setminus \{\text{div}(s)\}$, the height is given by:
\begin{equation} \label{eqheight}
	h_L(z) = \sum_{v \in M_K} N_v \log \|s(z)\|_{st,v}^{-1}
\end{equation}
where $N_v = [K_v : \mathbb{Q}_v]/[K : \mathbb{Q}]$.
Let $\varphi: \mathbb{P}^1 \to \mathbb{P}^1$ be a rational function of degree $d \geq 2$ defined over $K$, and let $\epsilon : \mathcal{O}(d) \overset{\sim}{\to} \varphi^* \mathcal{O}(1)$ be a $K$-isomorphism, known as a $K$-polarization. The canonical adelic metric associated to $(\varphi, \epsilon)$ on $\mathcal{O}(1)$ is then the family $\|\cdot\|_{\varphi, \epsilon} = (\|\cdot\|_{\varphi, \epsilon, v})$ defined for all $v \in M_K$. The canonical height relative to $\varphi$ is defined by:
\begin{equation} \label{eqheight1}
	h_{\varphi}(z) = \sum_{v \in M_K} N_v \log \|s(z)\|_{\varphi, \epsilon, v}^{-1}
\end{equation}
for all $z \in \mathbb{P}^1(K) \setminus \{\text{div}(s)\}$.
Let $\alpha \in K$ and define the section $s(z) = z_1 - \alpha z_2$ of $\mathcal{O}(1)$. Define $L_{\alpha}$ to be the adelic line bundle where, for each place $v \in M_K$, the metric is given by $\log \|s(z)\|_{st, v}^{-1} = \lambda_{\alpha, v}(z)$. 
Let $\mathcal{P} \subset \mathbb{P}^1(K)$ be a finite $\mathrm{Gal}(\overline{K}/K)$-invariant set, and let $s$ be a section of $\mathcal{O}(1)$ such that $\mathrm{div}(s) \notin \mathcal{P}$.
The height of $\mathcal{P}$ with respect to the adelic line bundle $L_{\alpha}$ is defined by $$h_{\mathcal{L}_{\alpha}}(\mathcal{P}) = \frac{1}{|\mathcal{P}|}
\sum_{z \in \mathcal{P}} \sum_{v \in M_K} N_v \log \|s(z)\|_v^{-1},$$ where $N_v = \frac{[K_v:\mathbb{Q}_v]}{[K:\mathbb{Q]} }$.

For a rational map $\varphi : \mathbb{P}^1 \to \mathbb{P}^1$  defined over $K$
with degree $d \ge 2$, let $\mathcal{L}_\varphi$ denote the canonical adelic line bundle associated to $\varphi$ (see Section~3.5 of \cite{petsche2012}). The associated height function $h_{L_\varphi}$ coincides with the canonical
height $h_\varphi$.

For any two rational maps $\varphi : \mathbb{P}^1 \to \mathbb{P}^1$ and $\psi : \mathbb{P}^1 \to \mathbb{P}^1$, defined over $K$, $L_\varphi$ and $L_\psi$ be the adelic line bundle associated to $\varphi$ and $\psi$ respectively. Let $s, t \in \Gamma(\mathbb{P}^1, \mathcal{O}(1))$ be two sections with $\operatorname{div}(s) \neq \operatorname{div}(t)$. We define the local Arakelov--Zhang pairing of $\mathcal{L}_\varphi$ and $\mathcal{L}_\psi$, with respect to the sections $s$ and $t$, by
\begin{align*}
    \langle \mathcal{L}_\varphi, \mathcal{L}_\psi \rangle_{s, t, v} &= - \int \left\{ \log \| s(z) \|_{\varphi,\epsilon_{\varphi}} \right\} \, d\Delta \left\{ \log \| t(z) \|_{\psi,\epsilon_{\psi}} \right\} \\
    &= \log \| s(\operatorname{div}(t)) \|_{\varphi,\epsilon_{\varphi}} - \int \log \| s(z) \|_{\varphi,\epsilon_{\varphi}} \, d\mu_{\psi}(z),
\end{align*}
where $\epsilon_{\varphi}$ and $\epsilon_{\psi}$ are any polarizations of
$\varphi$ and $\psi$, respectively and $\Delta \{-\log \|s(z)\|_{\varphi, \epsilon}\}=\del_{\mathrm{div}(s)}(z)-\mu_\varphi(z)$. Note that $\langle \mathcal{L}_\varphi, \mathcal{L}_\psi \rangle_{s, t, v}$ does not depend on the choice of polarizations $\epsilon_{\varphi}$ and $\epsilon_{\psi}$. The global Arakelov-Zhang pairing is then defined as 
\begin{align*}
\langle \mathcal{L}_\varphi, \mathcal{L}_\psi \rangle &= \sum_{v \in M_K} N_v \, \langle \mathcal{L}_\varphi, \mathcal{L}_\psi \rangle_{s,t,v} + h_{\mathcal{L}_{\varphi}}\!\left( \operatorname{div}(t) \right) + h_{\mathcal{L}_{\psi}}\!\left( \operatorname{div}(s) \right)\\
&=\sum_{v \in M_K} N_v \left( - \int \log \| s(z) \|_{\varphi,\epsilon_{\varphi},v} \, d\mu_{\psi,v}(z) - \log \| t(\operatorname{div}(s)) \|_{\psi,\epsilon_{\psi},v} \right).
\end{align*} 
One of the key result we require related to the Arakelov-Zhang pairing is the following.
\begin{theorem}[\cite{petsche2012}, Corollary 12]\label{lemmapairing}
    Let $(z_n)_{n\geq 0}$ be a sequence of distinct points in $\mathbb{P}^1(\overline{K})$ such that $h_{\mathcal{L}_{\psi}}(z_n) \to 0$ then $h_{\mathcal{L}_{\varphi}}(z_n) \to \langle \mathcal{L}_\varphi, \mathcal{L}_\psi \rangle$.
\end{theorem}
\subsection{Quantitative equidistribution}
At first we introduce some definition following \cite{favre}. Let $K$ be a number field, $K_v$ its completion for a place $v$ and $\C_v$ the completion of $\overline{K}_v$. Let $M_K$ denote the places of $v$. For each $v\in M_K$, let $\rho_v$ be a measure on $\mathbb{P}_{\mathrm{Berk}, v}^1$. We say that $\rho_v$ has a continuous potential if $\rho_v = \lambda_v+\Delta g$ for some continuous function $g$ which is the difference of two continuous subharmonic functions. Here note that $\lambda_v$ is the Gauss point $\zeta_{0,1}$ for non-archimedean $v$ and is uniform probability measure supported on the unit circle for archimeden $v$. We say that $\rho_v$ has a H\"older-continuous potential with exponent $\kappa$ with respect to a metric $d$ if there exists a constant $C>0$ such that 
\[|g(z) - g(w)|\leq C d(z, w)^{\kappa}\] for all classical points $z, w\in \mathbb{P}^1(\C_v)$.  

We call $\rho = (\rho_v)_{v\in M_K}$ an {\em adelic measure} if for each $v\in M_K,$ if each $\rho_v$ has a continuous potential and $\rho_v = \lambda_v$ for all but finitely many $v$. Given two measures $\rho_v, \sigma_v$ on $\mathbb{P}_{\mathrm{Berk}, v}^1$, we define a bilinear form
\[(\rho_v, \sigma_v)_v = -\int_{\mathbb{A}_{\mathrm{Berk},v}^1\times \mathbb{A}_{\mathrm{Berk},v}^1 \setminus{\mathrm{Diag}}} \log |x-y|_v d\rho_v(z)d\sigma_v(y)\]
where $\mathbb{A}_{\mathrm{Berk}, v}^1$ denote the Berkovich affine line over $\C_v$ and $\mathrm{Diag} = \{(x, x)\mid x\in \C_v\}$ is the diagonal classical points. This integral exists if both $\rho_v, \sigma_v$ either have a continuous potential or are probability measures supported on a finite subset of $\mathbb{P}^1(\overline{K})$.

For $\alpha \in \mathbb{P}^1(\overline{K})$, let $G_K(\alpha) = \mbox{Gal}(\overline{K}/K)$-conjugates of $\alpha$ and $[\alpha]$ denote the probability measure supported equally on the Galois conjugates of $\alpha$ over $K$, that is,
\begin{equation}
[\alpha] = \frac{1}{|G_K(\alpha)|}\sum_{z\in G_K(\alpha)}\delta_z
\end{equation}
where $\delta_z$ denotes the Dirac measure at $z$, which for $z\in \mathbb{P}^1(\overline{K})$ we interpret as the adelic measure of the point mass at $z$ each place. Then the canonical height $h_{\rho}: \mathbb{P}^1(\overline{K}) \to \mathbb{R}$ associated to $\rho$ is defined to be
\begin{equation}
h_{\rho}(\alpha) = \frac{1}{2}\sum_{v\in M_K} ((\rho-[\alpha], \rho-[\alpha]))_v
\end{equation}
where $((\cdot, \cdot))_v = N_v(\cdot, \cdot)_v$ with $N_v= [K_v:\Q_v]/[K:\Q]$. Note that for the standard measure, $((\lambda, \lambda))_v = (([\alpha], [\alpha]))_v =0$ and $((\lambda, [\alpha]))_v = \log^{+}\|\alpha\|_v$, and so $h_{\lambda} =h$ coincides with the usual absolute logarithmic Weil height. From \cite[Theorem 4]{favre}, if $\rho = (\rho_{\varphi, v})_{v\in M_K}$ is the adelic set of canonical measure associated to iteration of a rational map $\varphi$, then $h_{\rho} = h_{\varphi}$ is the usual Call-Silverman dynamical height. 

The set $\mathbb{H}_v:=\mathbb{P}^1_{\mathrm{Berk},v} (\mathbb{C}_v)\setminus \mathbb{P}^1(\mathbb{C}_v)$ is called the \emph{hyperbolic space} over $\mathbb{C}_v$. For an infinite place $v$, we say that a continuous function
$f : \mathbb{P}^1(\mathbb{C}_v) \to \mathbb{R}$ is of class $\mathcal{C}^k_{\mathrm{sph}}$
if it is $\mathcal{C}^k$ with respect to the spherical metric
$$d_{\mathrm{sph}}(x,y)=\frac{|x_1y_2 - x_2y_1|_v}
{\sqrt{|x_1|_v^2+|x_2|_v^2}\,\sqrt{|y_1|_v^2+|y_2|_v^2}},$$ where $x=(x_1:x_2)$ and $y=(y_1:y_2)$. For a finite place $v$, we say that
$f : \mathbb{P}^1_{\mathrm{Berk},v} \to \mathbb{R}$ is of class
$\mathcal{C}^k_{\mathrm{sph}}$ if it is locally constant outside of a finite subtree
$T \subset \mathbb{H}_v$, and $T$ is a finite union of segments on which
$f$ is of the usual class $\mathcal{C}^k$.

Given $f$ of class $\mathcal{C}_{\mbox{sph}}^k$ for $k\geq 1$, we define
$$\langle f, f\rangle_v = \int_{\C} \left(\frac{\partial f}{\partial x}\right)^2+\left(\frac{\partial f}{\partial y}\right)^2dxdy $$ if $v$ is archimedean. If $v$ is non-archimedean, we fix a base point $S_0\in \mathbb{H}_v$ and let $\partial f (S)$ be the derivative of $f$ restricted to the segment $[S_0, S]$. Then we define
\begin{equation} \label{Dpro}
\langle f, f\rangle _v = \int_{\mathbb{P}^1(\mathbb{C}_v)} (\partial f )^2 d\lambda.
\end{equation}
We can now state the quantitative equidistribution result of Favre-Rivera-Litelier (\cite{favre}).

\begin{proposition}\label{quantequi}
 Let $\rho = (\rho_v)_{v \in M_K}$ be an adelic measure where each $\rho_v$ has Hölder-continuous potentials of exponent $\kappa \leq 1$ with respect to the spherical metric. Fix any $\delta > 0$. Then there exists a constant $C_2 > 0$, depending only on the $\rho$ such that for all places $v$ and all functions $f$ of class $\mathcal{C}^1_{\mathrm{sph}}$ on $\mathbb{P}^1_{\mathrm{Berk},v}$, and for all finite $\emph{Gal}(\overline{K}/K)$-invariant sets $\mathcal{P}$, we have
\begin{equation}\label{equid}
    \left| \frac{1}{|\mathcal{P}|} \sum_{z \in \mathcal{P}} f(z) - \int_{\mathbb{P}^1_{\mathrm{Berk},v}} f \, d\rho_v \right|_v
\leq \frac{\mathrm{Lip_{sph}}(f)}{|\mathcal{P}|^{1/\kappa}} + \left(2 h_\rho(\mathcal{P}) + C_2 \frac{\log |\mathcal{P}|}{\sqrt{|\mathcal{P}|}}\right)^{1/2} \langle f,f\rangle _v^{1/2},
\end{equation}
where $\mathrm{Lip_{sph}}(f)$ is the Lipschitz constant for $f$ with respect to the spherical metric. 
\end{proposition}
\begin{proof}
    See Theorem 7 of \cite{favre}.
\end{proof}
Next we derive a similar bound to that in \eqref{equid} for $f(z) = \lambda_{\tau,v}(z)$, where $0 < \tau <1$ and $\lambda_{\tau,v}(z):=\log^+|z|_v+\log^+|\alpha|_v-\log \max\{\tau, |z-\alpha|_v\}$.

\begin{proposition}\label{corollary2.2}
\emph{(}Keep the notation $K, v, \mathcal{P}$ and $\kappa$ as above.\emph{)}  Let $\varphi$ be a rational map of degree $d \geq 2$ defined on $\mathbb{P}^1(K)$ and let $(\mu_{\varphi,v})_v$ be an adelic measure where each equlibrium measure $\mu_{\varphi,v}$ has Hölder-continuous potentials of exponent $\kappa \leq 1$ with respect to the spherical metric. Then for all finite $\emph{Gal}(\overline{K}/K)$-invariant sets $\mathcal{P}$, there exists a constant $C_3 >0$ (depending on $\varphi$) such that 
\begin{equation*}
    \left| \frac{1}{|\mathcal{P}|} \sum_{z \in \mathcal{P}} \lambda_{\tau,v}(z) - \int \lambda_{\tau,v}(z) d\mu_{\varphi,v} \right|_v \leq C_3 \left( \frac{1}{|\mathcal{P}|^{1/2}} + \left( h_\varphi(\mathcal{P}) + \frac{\log|\mathcal{P}|}{|\mathcal{P}|} \right)^{1/2} \right).
\end{equation*}
\end{proposition}
\begin{proof}
   Let $v$ be an archimedean place of $K$ and identify $\mathbb{C}_v$ with $\mathbb{C}$.
   Consider the function 
   \begin{equation*}
       \lambda_{\alpha, v}(z)=\log^+|z|_v+\log^+|\alpha|_v-\log |z-\alpha|_v.
   \end{equation*}
   For any real number $0< \tau< 1$, we truncate our function to 
   \begin{equation}\label{truncation}
       \lambda_{\tau, v}(z)=\log^+|z|_v+\log^+|\alpha|_v-\log \max\{\tau, |z-\alpha|_v\}
   \end{equation}
and hence
   $$\lambda_{\tau,v}(z)=
\begin{cases}
\log|z|_v+\log|\alpha|_v-\log\tau, & |z-\alpha|_v \le \tau,\\
\log|z|_v+\log|\alpha|_v-\log|z-\alpha|_v, & |z-\alpha|_v>\tau.
\end{cases}$$
Then $\lambda_{\tau,v}$ is Lipschitz continuous on $\mathbb{C}$ with Lipschitz constant Lip($\lambda_{\tau,v})=1+\frac{1}{\tau}$. We consider the usual charts of the complex projective line
$(U_0,u_0)$ and $(U_1,u_1)$, where the open subsets are
$$
U_0 := \{(1:z)\in \mathbb{P}^1(\mathbb{C}) : z\in \mathbb{C}\},
\qquad
U_1 := \{(z:1)\in \mathbb{P}^1(\mathbb{C}) : z\in \mathbb{C}\},
$$
and the homeomorphisms
$$
u_0 : U_0 \longrightarrow \mathbb{R}^2,
\qquad
u_1 : U_1 \longrightarrow \mathbb{R}^2,
$$
$$
(1:z) \longmapsto (x,y),
\qquad
(z:1) \longmapsto (x,y),
$$
where $z = x + iy \in \mathbb{C}$. 
    Now, the Dirichlet form can be calculated in charts as  
    \begin{equation*}
        \langle \lambda_{\tau,v},\lambda_{\tau,v}\rangle_v=\int_{\overline{D}(0,1)} \left(\frac{\partial \lambda_{\tau, v,0}}{\partial x}\right)^2+\left(\frac{\partial \lambda_{\tau, v,0}}{\partial y}\right)^2 dx dy + \int_{D(0,1)}\left(\frac{\partial \lambda_{\tau, v,1}}{\partial x}\right)^2+\left(\frac{\partial \lambda_{\tau, v,1}}{\partial y}\right)^2 dx dy,
    \end{equation*}
   where $\lambda_{\tau, v,0}(x,y)=\lambda_{\tau,v}(x+iy:1), \lambda_{\tau, v,1}(x,y)=\lambda_{\tau, v} (1: x+iy)$.
   On a chart $U_i$ with coordinates $x$ and $y$, where $u_i(\alpha)=(a_i,b_i)$ and 
   \begin{align*}
         \lambda_{\tau,v} \circ u_i^{-1}(x,y)=&\log \max\{1, \sqrt{x^2+y^2}\}+\log \max\{1, \sqrt{a_i^2+b_i^2}\}\\ &-\log \max\{\tau, \sqrt{(x-a_i)^2+(y-b_i)^2}\},
   \end{align*}
    which has weak partial derivatives
   \begin{align*}
        \frac{\partial \lambda_{\tau,v}}{\partial x}&=\mathbb{I}\{\sqrt{x^2+y^2} >1\}\frac{x}{x^2+y^2}-\mathbb{I}\{\sqrt{(x-a_i)^2+(y-b_i)^2} >\tau\}\frac{x-a_i}{(x-a_i)^2+(y-b_i)^2},\\
        \frac{\partial \lambda_{\tau,v}}{\partial y}&=\mathbb{I}\{\sqrt{x^2+y^2} >1\}\frac{y}{x^2+y^2}-\mathbb{I}\{\sqrt{(x-a_i)^2+(y-b_i)^2} >\tau\}\frac{y-b_i}{(x-a_i)^2+(y-b_i)^2},
   \end{align*} 
   where $\mathbb{I}$ is an indicator function.
Then by making a substitution to move $(a_i, b_i)$ to the origin
   \begin{equation*}
       \langle \lambda_{\tau,v},\lambda_{\tau,v}\rangle_v=2\int_0^{2\pi}\int_\tau^1 \frac{dr d\theta}{r}= -4 \pi \log \tau.
   \end{equation*}
   For $v$ non-archimedean, $0<\tau<1$, and $\alpha \in \mathbb{C}_v$, we have 
\begin{equation*}
\lambda_{\tau,v}(z) = \log \max \mathrm{diam} (z) +\log \max \mathrm{diam}(\alpha)- \log\max\{\tau, \mathrm{diam}(z \vee \alpha)\}.
\end{equation*}
which extends the function $\lambda_{\tau,v}(z) = \log^+|z|_v + \log^+|\alpha|_v- \log\max\{\tau,|z-\alpha|_v\}$ on $\mathbb{P}^1(\mathbb{C}_v) $ to $\mathbb{P}^1_{\mathrm{Berk}}(\mathbb{C)}$.
Let $\xi := \zeta_{\alpha,1}, \Lambda_\alpha := [\zeta_{\alpha,\tau},\xi]$ and $\Lambda_0=\{\zeta_{0,1}\} $. Now, define $\Lambda:=\Lambda_\alpha \cup \Lambda_0$. 
Then $\lambda_{\tau,v}$ is locally constant outside of $\Lambda$. Indeed, each term in $\lambda_{\tau,v}$ is constant on the connected components of $\mathbb{P}^1_{\mathrm{Berk}}(\mathbb{C}_v) \setminus \Lambda$, so that 
\begin{equation*}
\lambda_{\tau,v}(z) = \lambda_{\tau,v}\big(z \,\vee_{\xi}\, \zeta_{\alpha,\tau}\big).
\end{equation*}
Moreover, on $\Lambda$ the function $\lambda_{\tau,v}$ is $\mathcal{C}^1$ and the directional derivative along the path from $\xi$ to $\zeta_{\alpha,\tau}$ and at $\zeta_{0,1}$ is constant, with 
\begin{equation*}
\partial \lambda_{\tau,v}(z) = 
\begin{cases}
1, & \text{on } [\zeta_{\alpha,\tau},\zeta_{\alpha,1}],\\
0 \  & \text{otherwise}.
\end{cases}
\end{equation*}
It follows that $\lambda_{\tau,v}$ is Lipschitz, with 
\begin{equation*}
\operatorname{Lip}(\lambda_{\tau,v}) =1.
\end{equation*}
Finally, using the non-archimedean Dirichlet pairing formula given in \eqref{Dpro}, we obtain 
\begin{equation*}
\langle \lambda_{\tau,v},\lambda_{\tau,v}\rangle_v =-\log \tau.
\end{equation*}
Then substituting the values of $ \mathrm{Lip}(\lambda_{\tau,v})$ and $\langle \lambda_{\tau,v},\lambda_{\tau,v}\rangle_v$ in \eqref{equid} and by choosing a constant $C_3 >0$, we get 
\begin{align}\label{equi.bound} 
    \left| \frac{1}{|\mathcal{P}|} \sum_{z \in \mathcal{P}} \lambda_{\tau,v}(z) - \int \lambda_{\tau,v}(z) d\mu_{\varphi,v} \right|_v
        &\leq \left( \frac{2}{|\mathcal{P}|^{1/2}} + \left( 2h_\varphi(\mathcal{P}) + C_2\frac{\log|\mathcal{P}|}{|\mathcal{P}|} \right)^{1/2} \right)  O(\log \tau) \nonumber \\
        &\leq C_3 \left( \frac{1}{|\mathcal{P}|^{1/2}} + \left( h_\varphi(\mathcal{P}) + \frac{\log|\mathcal{P}|}{|\mathcal{P}|} \right)^{1/2} \right).
\end{align}
\end{proof}
\begin{proposition}\label{propadelic}
Let $K$ be a number field, $\varphi(z) = z^d$ be a rational map of degree $d \geq 2$ defined on $\mathbb{P}^1(K)$ and let $\mathcal{L}_{\varphi}$ be a canonical adelization for $\mathcal{O}(1)$ over $K$. Let $\alpha \in \overline{K}$ and $\beta$ be a fixed non-zero element in $K$. Let $(z_n)_{n\geq0}$ be a sequence of distinct points in $\mathcal{O}_\varphi^-(\beta)$ and for each $n$ let $\mathcal{P}_n$ be the set of $\mathrm{Gal}(\overline{K}/K)$ conjugates of $z_n$.  Then for any adelic line bundle $\mathcal{L}$ there exists a constant $C_{AZ, d}>0$ (depending only on $d$ such that 
	$$\left| h_{\mathcal{L}}(\mathcal{P}_n) - \langle \mathcal{L}, \mathcal{L}_{\varphi} \rangle \right|_v \leq C_{AZ,d}\left(\frac{1+\log |\mathcal{P}_n|^{1/2}}{|\mathcal{P}_n|^{1/2}}\right).$$
    \end{proposition} 
	\begin{proof}
 Suppose that $z_n \in \mathcal{O}_\varphi^-(\beta)$ and  $\mathcal{P}_n$ be any $\mathrm{Gal}(\overline{K}/K)$-orbit of $z_n$. The height $(z_n)$ associated to the adelic line bundle $\mathcal{L}$ is given by
\begin{equation}\label{prop2.3eq2}
h_{\mathcal{L}}(z_n) = \frac{1}{|\mathcal{P}_n|}\sum_{z\in \mathcal{P}_n} h_{\mathcal{L}}(z)=\frac{1}{|\mathcal{P}_n|}\sum_{z\in \mathcal{P}_n} 
\sum_{v \in M_K} N_v \log \| s(z) \|_{\mathrm{st},v}^{-1}.
\end{equation} 
    Since $$h_\varphi(z_n) = \frac{1}{|\mathcal{P}_n|}\sum_{z\in \mathcal{P}_n}h_\varphi(z) =\frac{1}{|\mathcal{P}_n|}\sum_{z\in \mathcal{P}_n}\lim_{m \to \infty} \dfrac{h(\varphi^m(z))}{d^m}$$ and $z \in \mathcal{O}_\varphi^-(\beta)$, it follows that $h_\varphi(z_n) \to 0$. Then by Theorem \ref{lemmapairing}, we have
\begin{equation}\label{prop2.3eq1}
\langle \mathcal{L}, \mathcal{L}_{\varphi} \rangle = \lim_{n \to \infty} h_{L}(z_n).
\end{equation} 
For each place $v \in M_K$, we have the local height function 
$\lambda_{\tau,v}(z) = \log \| s(z) \|_{\mathrm{st},v}^{-1}$.
Then from \eqref{prop2.3eq1} and \eqref{prop2.3eq2},
\begin{equation*}
\langle \mathcal{L}, \mathcal{L}_{\varphi} \rangle = \lim_{n \to \infty} \frac{1}{|\mathcal{P}_n|} \sum_{z \in \mathcal{P}_n} \sum_{v \in M_K} N_v \lambda_{\tau,v}(z).
\end{equation*} 
Let $S$ be a finite set of places including all of the archimedean places. Since $\lambda_{\tau,v}$ vanishes identically for all places $v \notin S$,
the sum over $M_K$ reduces to a sum over the finite set $S$, yielding
\begin{equation*}
\langle \mathcal{L}, \mathcal{L}_{\varphi} \rangle=\lim_{n \to \infty} \frac{1}{|\mathcal{P}_n|} \sum_{z \in \mathcal{P}_n}
\sum_{v \in S} N_v \lambda_{\tau,v}(z).
\end{equation*} 
Because the set $S$ is finite and independent of $n$, the limit may be
interchanged with the sum over $v \in S$. Since $h_\varphi(z_n) \to 0$ so by the equidistribution result, $(\mathcal{P}_n)_{n\geq 0}$ equidistribute with respect to equilibrium measure  $\mu_{\varphi,v}$, so
\begin{equation}\label{eqprop2.3}
\langle \mathcal{L}, \mathcal{L}_{\varphi} \rangle = \sum_{v \in S} N_v \int \lambda_{\tau,v}(z) \, d\mu_{\varphi,v}.
\end{equation}
For any $\mathrm{Gal}(\overline{K}/K)$-invariant set $\mathcal{P}_n$, we have $h_{\mathcal{L}}(\mathcal{P}_n)= h_{\mathcal{L}}(z_n)$. Then from \eqref{prop2.3eq2} and \eqref{eqprop2.3}, we get
    $$\left| h_{\mathcal{L}}(\mathcal{P}_n) - \langle \mathcal{L}, \mathcal{L}_{\varphi} \rangle \right|_v
		= \sum_{v \in M_K} N_v \left| \frac{1}{|\mathcal{P}_n|} \sum_{z \in \mathcal{P}_n} \lambda_{\tau,v}(z) - \int \lambda_{\tau,v}(z) \, d\mu_{\varphi,v} \right|_v.$$
By Proposition \ref{corollary2.2}, we get
    \begin{align}\label{eq4.2}
        \left| h_{\mathcal{L}}(\mathcal{P}_n) - \langle \mathcal{L}, \mathcal{L}_{\varphi} \rangle \right|_v
        & \leq   \sum_{v \in M_K} N_v C_3 \left( \frac{1}{|\mathcal{P}_n|^{1/2}} + \left(  \frac{\log|\mathcal{P}_n|}{|\mathcal{P}_n|} \right)^{1/2} \right) \nonumber\\ 
        & \leq  \sum_{v \in M_K} N_v C_3   \left( \frac{1+\log|\mathcal{P}_n|^{1/2}}{|\mathcal{P}_n|^{1/2}}  \right),
    \end{align}
  where $C_3$ depends on $d$.  This completes the proof of Proposition \ref{propadelic}.
\end{proof}	

\subsection{Linear forms in logarithms} We also require the theory of linear forms in logarithms, initially developed by Baker \cite{alan}.  This theory provides lower bounds for expressions of the form
\begin{equation}
	\left| a_1^{b_1}\cdots a_n^{b_n}-1\right|
\end{equation}
in terms of the heights of $a_i$ and $b_i$.
The first key result is a consequence of a theorem by Laurent, Mignotte, and Nesterenk \cite{LMN}.
\begin{theorem}[{\cite{yap}, Corollary 3.2}]\label{linearform} Let $\alpha$ be an algebraic number with $|\alpha|=1$. Then for any $\epsilon > 0$, then there exists a constant $C_\epsilon >0$ such that for any roots of unity $\zeta$ of order $n$, we have 
	$$ \log|\zeta - \alpha| \geq -C_\epsilon[\mathbb{Q}(\alpha):\mathbb{Q}]^3(h(\alpha)+1)n^\epsilon.$$
\end{theorem}
\subsection{Quantitative logarithmic equidistribution}
Let $\varphi$ be a rational map. Lyubich \cite{Lyubich} proved that any sequence of distinct points $(x_n)_{n\ge 1}$ contained in the backward orbit $O^-_\varphi(\beta)$ is equidistributed on $\mathbb{P}^1(\mathbb{C})$ with respect to the Haar measure. In addition, this sequence is logarithmically equidistributed at a fixed point $\alpha$. In particular, when the canonical height satisfies $h_\varphi(\alpha) > 0$, it follows that only finitely many points in the backward orbit $\mathcal{O}^-_\varphi(\beta)$ can be $S$-integral relative to $\alpha$.

To get a uniform result, we suppose that for a fixed place $v \in M_K $, let $\mathcal{P}$ be the Galois orbit of some points in $\mathcal{O}^-_\varphi(\beta)$, and $S$ is a finite set of places of $K$ containing all archimedean ones and $v\in M_K$. For $|\mathcal{P}|$ large enough, we will try to find an  upper bound of the form 
$$\sum_{v\in S}\left|\frac{1}{|\mathcal{P}|} \sum_{z\in \mathcal{P}} N_v\lambda_{\alpha, v}(z) - \int N_v\lambda_{\alpha, v}(z)d\mu_{\varphi,v}\right|_v< \frac{1}{2} h(\alpha).$$
But by the Arakelov-Zhang pairing, we must have
$$\lim_{|\mathcal{P}|\to \infty} \frac{1}{|\mathcal{P}|} \sum_{v\in S}\sum_{z\in \mathcal{P}} N_v\lambda_{\alpha, v}(z) = h_{\varphi}(\alpha) + \sum_{v\in M_K} \int N_v \lambda_{\alpha, v}(z)d\mu_{\varphi,v}$$ and so if $h_{\varphi}(\alpha) > \frac{1}{2} h(\alpha)$, which is true for $\alpha$ of large height, we obtain that $\mathcal{P}$ can not be $S$-integral relative to $\alpha$.
\section{Finiteness of S-integral points}\label{sec-finiteness}
In this section, we will provide a proof of finiteness of $S$-integral points in backward orbits, which was earlier proved by Sookdeo \cite[Theorem 1.3]{Sookdeo}. Note that Sookdeo proved this result using the idea that the number of Galois orbits for $z^n-\beta$ is bounded below when $\beta$ is not $0$ or a root of unity and Siegel’s theorem for integral points on $\mathbb{G}_m(K)$. Precisely,
\begin{proposition}\label{prop01}
Let $K$ be a number field and let $S$ be a finite set of places of $K$ that contain all the archimedean places.	Suppose $\alpha \not \in \emph{PrePer}(\varphi, \overline{K})$, and $\beta$ be a fixed non-zero element in $K$. Then there are at most finitely many points in $\mathcal{O}^-_\varphi(\beta)$ which are $S$-integral relative to $\alpha$.
\end{proposition}
To prove this, we need the following lemma.
\begin{lemma}\label{lemma1}
	Let $\varphi(z)=z^d$ be a rational map of degree $d \geq 2$. Fix a non-archimedean place $v$ of $K$ corresponding to the prime $p$. Let $\alpha \in \mathbb{P}^1(\overline{K})$ with $|\alpha|_v=1$, and let $\beta$ be a fixed non-zero element in $K$. Then for each real number $r$ with $0 <r<1$,
	\begin{enumerate}
		\item [\textup{(i)}] there are at most finitely $\gamma \in \overline{K}_v$ such that $\gamma^{d^n}=\beta$ for some positive integer $n$ satisfying $|\gamma - \alpha|_v<r$ and
		\item[\textup{(ii)}]there is a bound $M(\alpha) >0$ such that $|\gamma - \alpha|_v \geq M(\alpha)$ for all $\gamma \in \overline{K}_v$ with $\gamma \in \mathcal{O}^-_\varphi(\beta)$. 
	\end{enumerate}
\end{lemma}
\begin{proof}
	Note that if $\gamma$ and $\gamma'$ are elements in $\overline{K}_v$ such that $\gamma^{d^n}=\beta$ and $\gamma'^{d^n}=\beta$ for some positive integer $n$ with $|\gamma - \alpha|_v<r$
	and $|\gamma'-\alpha|_v<r$,	then 
	\begin{align}\label{eq4.2new}
		|\gamma-\gamma'|_v=|\gamma - \alpha-( \gamma'-\alpha)|_v\leq \max\{|\gamma - \alpha|_v, |\gamma'-\alpha|_v\}<r.
	\end{align}
	Since $|\alpha|_v=1$ and if $|\gamma|_v>1$ then $|\gamma - \alpha|_v=\max \{|\alpha|_v,|\gamma|_v\}<r<1$, which is a contradiction. If $|\gamma|_v<1$, then similarly we will get a contradiction. Hence, we deduce that $|\gamma|_v= 1$. So, from \eqref{eq4.2new}, we get $|1-\gamma^{-1}\gamma'|_v<r$.
	Moreover, $\gamma^{-1}\gamma'=\zeta$, where $\zeta \in \overline{K}_v$ is a root of unity. Then by Lemma 1.1 of \cite{baker2008}, there are finitely many $\zeta \in \overline{K}_v$ such that $|1-\zeta|_v<r$. Consequently, there are finitely many $\gamma \in \overline{K}_v$ such that $|\gamma - \alpha|_v<r$. This completes the proof of (i).  Assume there are $m$ number of $\gamma \in \overline{K}_v$ such that $\gamma^{d^n}=\beta$ for some positive integer $n$ satisfying $|\gamma - \alpha|_v<r$. Put $M(\alpha)=\inf_{1 \leq i \leq m} |\gamma_i-\alpha|_v$. Then (ii) follows immediately.
\end{proof}

\subsection{Proof of Proposition \ref{prop01}}
    By replacing $K$ with $K(\alpha)$, and $S$ with the set of places $S_{K(\alpha)}$ lying over $S$, the problem reduces to proving the theorem for $\alpha \in K$. Indeed, if $\gamma \in \mathcal{O}^-_\varphi(\beta)$ and is $S$-integral with respect to $\alpha$ over $K$, then each $K$-conjugate of $\gamma$ is $S_{K(\alpha)}$-integral with respect to $\alpha$ over $K(\alpha)$. Suppose $\alpha \not \in \mbox{PrePer}(\varphi, \overline{K})$ and there are infinitely many distinct $\gamma_n \in \mathcal{O}^-_\varphi(\beta), n\geq 1$  which are $S$-integral with respect to $\alpha$.  To get the contradiction, we will evaluate the sum  
$$T_n = \frac{1}{[K(\gamma_n) : \mathbb{Q}]} \sum_{v \text{ of } K} \sum_{\sigma: K(\gamma_n)/K \hookrightarrow \overline{K}_v} \log |\sigma(\gamma_n) - \alpha|_v$$ in two different ways: On the one hand, we obtain $T_n = 0$, for all $n$ by using product formula and on the other hand, $\lim_{n \rightarrow \infty}T_n = h(\alpha) > 0$ using $S$-integrality and equidistribution. This contradiction will imply that there are finitely many distinct elements $\gamma_n \in \mathcal{O}^-_\varphi(\beta)$ which are $S$-integral with respect to $\alpha$. 

For all $n$,
\begin{align*}
T_n &= \frac{1}{[K(\gamma_n) : \mathbb{Q}]} \sum_{v \text{ of } K} \sum_{\sigma: K(\gamma_n)/K \hookrightarrow \overline{K}_v} \log |\sigma(\gamma_n) - \alpha|_v \\
	&= \frac{1}{[K(\gamma_n) : \mathbb{Q}]} \sum_{v \text{ of } K} \sum_{w \mid v} \log |\gamma_n - \alpha|_w  
	= \frac{1}{[K(\gamma_n) : \mathbb{Q}]} \sum_{w \text{ of } K(\gamma_n)} \log |\gamma_n - \alpha|_w \\
	&= \frac{1}{[K(\gamma_n) : \mathbb{Q}]} \log \prod_{w \text{ of } K(\gamma_n)} |\gamma_n - \alpha|_w 
	= \log \left(\prod_{w \text{ of } K(\gamma_n)} |\gamma_n - \alpha|_w^{\frac{1}{[K(\gamma_n) : \mathbb{Q}]}}\right)  \\
	&= \log 1 = 0.
\end{align*}
Set $v \notin S$ and $\sigma: K(\gamma_n)/K \hookrightarrow \overline{K}_v$. We know, for $ v\in K\setminus S, \gamma_n$ is $S$-integral relative to $\alpha$ if each pair of $K$-embedding $\sigma: K(\gamma_n) \hookrightarrow \overline{K}_v, \sigma': K(\alpha) \hookrightarrow \overline{K}_v$, we have  $\|\sigma(\gamma_n), \sigma'(\alpha)\|_v = 1$ under the spherical metric on $\mathbb{P}^1(\overline{K}_v)$.  
Equivalently, since	$\gamma_n$ is $S$-integral relative to $\alpha$, we have  
$$\begin{cases}  
	|\sigma(\gamma_n) - \alpha|_v \geq 1, & \text{if } |\alpha|_v \leq 1 \\  
	|\sigma(\gamma_n)|_v \leq 1, & \text{if } |\alpha|_v > 1. 
\end{cases} $$
If $|\alpha|_v > 1$, then 	$
|\sigma(\gamma_n) - \alpha|_v = \max \{|\sigma(\gamma_n)|_v,|\alpha|_v \} = |\alpha|_v$  and if $|\alpha|_v \leq 1$, then $|\sigma(\gamma_n) - \alpha|_v = \max \{|\sigma(\gamma_n)|_v, |\alpha|_v \} = |\sigma(\gamma_n)|_v$. It follows that for each $v\not \in S$ and $n$ sufficiently large $$ \log |\sigma(\gamma_n)-\alpha|_v= \log \max \{|\sigma(\gamma_n)|_v, |\alpha|_v \}.$$ Now $|\sigma(\gamma_n)|_v$ remains unchanged when $n$ is fixed and $\sigma$ varies since $\sigma(\gamma^{d^n})=\beta$.
Therefore, $$\sum_{\sigma: K(\gamma_n)/K \hookrightarrow \overline{K}_v}\log |\sigma(\gamma_n)-\alpha|_v =[K(\gamma)_n:K]\cdot \log \max \{|\sigma(\gamma_n)|_v, |\alpha|_v \}. $$ Hence, for $v \notin S$,
\begin{align}
	&\frac{1}{[K(\gamma_n) : \mathbb{Q}]} \sum_{\sigma: K(\gamma_n)/K \hookrightarrow \overline{K}_v}\log |\sigma(\gamma_n) - \alpha|_v \\ \nonumber
    &=\frac{1}{[K(\gamma_n) : \mathbb{Q}]} \sum_{\sigma: K(\gamma_n)/K \hookrightarrow \overline{K}_v} \log \max \{ |\sigma(\gamma_n)|_v, |\alpha|_v\} \nonumber \\
	&= \frac{[K(\gamma_n) : K]}{[K(\gamma_n) : \mathbb{Q}]}  \log \max \{ |\sigma(\gamma_n)|_v, |\alpha|_v\} \nonumber \\&=\frac{1}{[K:\mathbb{Q}]}\log \max \{ |\sigma(\gamma_n)|_v, |\alpha|_v\}.
\end{align}
Since $\lim_{n\to \infty} |\sigma(\gamma_n)|_v=\lim_{n\to \infty} |\beta|_v^{1/d^n}=1$, we get  
\begin{align*}
	\lim_{n \to \infty}  \frac{1}{[K(\gamma_n) : \mathbb{Q}]}\sum_{\sigma: K(\gamma_n)/K \hookrightarrow \overline{K}_v} \log |\sigma(\gamma_n) - \alpha|_v&=\frac{1}{[K:\mathbb{Q}]}\lim_{n \to \infty}\log \max \{ |\sigma(\gamma_n)|_v, |\alpha|_v\}\\ &=\frac{1}{[K:\mathbb{Q}]}\log \max \{ 1, |\alpha|_v\}.
\end{align*}
 We rewrite $T_n$ as follows
\begin{equation*}
T_n = \frac{1}{[K(\gamma_n) : \mathbb{Q}]}
	\left(\sum_{v \in S}\sum_{\sigma: K(\gamma_n)/K \hookrightarrow \overline{K}_v}\log |\sigma(\gamma_n) - \alpha|_v + \sum_{v \notin S}\sum_{\sigma: K(\gamma_n)/K \hookrightarrow \overline{K}_v}  \log |\sigma(\gamma_n) - \alpha|_v\right),
\end{equation*}	
letting $n \to \infty $, and since $S$ is finite we can interchange the limits and sum ove $v \in S$.
\begin{align}\label{al3.1}
\begin{split}
0= \lim_{n\to \infty} T_n &=\sum_{v\in S}\lim_{n \to \infty}  \frac{1}{[K(\gamma_n) : \mathbb{Q}]}  \sum_{\sigma: K(\gamma_n)/K \hookrightarrow \overline{K}_v} \log |\sigma(\gamma_n) - \alpha|_v\\
	& +\frac{1}{[K:\mathbb{Q}]} \sum_{v\notin S} \log \max \{ 1, |\alpha|_v\}.
    \end{split}
\end{align}
Next we show that for each $v \in S$, 
\begin{equation}\label{eq1}
 \lim_{n \to \infty} \frac{1}{[K(\gamma_n) : \mathbb{Q}]} \sum_{\sigma: K(\gamma_n)/K \hookrightarrow \overline{K}_v} \log |\sigma(\gamma_n) - \alpha|_v=\frac{1}{[K:\mathbb{Q}]}\log \max \{ 1, |\alpha|_v\}.
\end{equation} 
Substituting \eqref{eq1} into \eqref{al3.1}, we get $h(\alpha)=0$, which is a contradiction.	

For $v \in S$, we will consider the nonarchimedean and archimedean cases separately. For nonarchimedean case, by \eqref{lemma1} let $N(r)$ be number of points $\sigma(\gamma_n)$ satisfying $0< |\sigma(\gamma_n)-\alpha| <r$ where $ 0 <r <1$. Then 
\begin{align*}
    0 &\geq \lim_{n\to\infty}
\frac{1}{[K(\gamma_n):\mathbb{Q}]}
\sum_{\sigma} \log |\sigma(\gamma_n)-\alpha|_v \\
&\ge \lim_{n\to\infty}
\frac{1}{[K(\gamma_n):\mathbb{Q}]}
\Big( ([K(\gamma_n):K]-N(r))\log r + N(r)\log M(\alpha) \Big) \\
&= \frac{\log r}{[K:\mathbb{Q}]}.
\end{align*}
Taking $ r \to 1$ gives
$$\lim_{n\to\infty} \frac{1}{[K(\gamma_n):\mathbb{Q}]} \sum_{\sigma: K(\gamma_n)/K \hookrightarrow \overline{K}_v} \log |\sigma(\gamma_n)-\alpha|_v = 0 = \log \max\{1,|\alpha|_v\}.$$
This concludes the proof of \eqref{eq1} when $v$ is nonarchimedean. For archimedean we can proceed similarly as in Baker, Ih and Rumely \cite{baker2008}, to obtain \eqref{eq1}.  \qed

\section{Proof of Main Results}\label{sec-proof}
At first, we need to prove the existence of a constant $A$ in terms a power of the degree $[K(\alpha):K]$ for the inequality 
\[\max_{\gamma \in \mathcal{P}} \log|\gamma - \alpha|_v^{-1} < A(h(\alpha)+h(s)+1) |\mathcal{P}|^{\frac{1}{2}-\delta}.\]

The following proposition will provide a bound on the $v$-adic logarithmic distance between a non-preperiodic point $\alpha$ and a point $\gamma \in \mathcal{O}^-_\varphi(\beta)$, which can be determined in terms of the height of $\alpha$, height of $s$ and the size of the Galois orbit $\mathcal{P}$. Such bounds are crucial for controlling the rate of equidistribution in backward orbits and understanding the distribution of $S$-integral points.

\begin{proposition}\label{prop3.1}
Let $K$ be a number field, $\varphi(z)=z^d$ be a rational map of degree $d\geq 2$ and $\beta$ be a fixed non-zero element in $K$. Let $\mathcal{P}$ be any $\emph{Gal}(\overline{K}/K)$-orbit of some points in $\mathcal{O}^-_\varphi(\beta)$ and let $v$ be an archimedean place of $K$. Let $\alpha \not \in \emph{PrePer}(\varphi, \overline{K}), \gamma \in \mathcal{O}_\varphi^-(\beta)$ with $s:=|\gamma|_v$ for some real $s$. Then for any $\epsilon > 0$, there exists a constant $C_\epsilon$ such that  
	$$ \max_{\gamma \in \mathcal{P}} \log |\gamma - \alpha|_v^{-1} < C_\epsilon[K:\mathbb{Q}]^3(h(\alpha)+h(s)+1)|\mathcal{P}|^\epsilon.$$ 
\end{proposition}
\begin{proof}
 First assume that $|\gamma|_v \neq 1$. For any $\gamma \in \mathcal{O}_\varphi^-(\beta)$ we have
	\begin{equation*}
		|\alpha|_v - s \leq |\gamma - \alpha|_v.
	\end{equation*}
	Now, $|\alpha|_v$ is a real algebraic number in some field $K'$ of degree at most 2 larger than $K$.
	Since $|\alpha|_v = \alpha \overline{\alpha}$, where $\overline{\alpha}$ is the complex conjugate, we have 
	$ h(|\alpha|_v^2) = h(\alpha \overline{\alpha}) =  h(\alpha) + h(\overline{\alpha}) = 2h(\alpha)$. Also, by height property
	$$
	h(|\alpha|_v - s) \leq h(|\alpha|_v) + h(-s) + \log 2.
	$$
	Therefore,
	\begin{equation}\label{h1eq}
		h(|\alpha|_v - s) \leq h(|\alpha|_v) + h(-s) + \log 2 \leq 2h(\alpha) +h(s)+ 2\log 2.
	\end{equation}
	If $s < 1$, then $ h(s)= \log \max \{|s|_v, 1\} = 0 $. Therefore, \eqref{h1eq} can be rewritten as $h(|\alpha|_v-s) \leq 2h(\alpha)+2\log 2$.
	Otherwise, $h(|\alpha|_v-s) \leq h(\alpha)+h(s)+2 \log 2$.
	In the first case,
	$$
	\log |\gamma - \alpha|_v^{-1}\leq 2(h(\alpha) + 1) \leq [K': \Q] (h(\alpha) + 1),$$ 
	and in the remaining cases 
	$$
	\log |\gamma - \alpha|_v^{-1}\leq 2(h(\alpha) + h(s)+ 1) \leq [K': \Q] (h(\alpha) + h(s)+ 1).$$
	Next assume that $|\gamma|_v=1$. Since $\gamma^{d^n}=\beta$ for some $n$, we may write 
		$\gamma=\beta^\frac{1}{d^n} e^\frac{2 \pi i}{d^n}  \text{ and }\alpha=\beta^\frac{1}{d^n} e^{i\theta_0}$. Thus,
	\begin{align*}
		\log|\gamma - \alpha|_v &=\log (|\beta^\frac{1}{d^n}|_v|e^\frac{2 \pi i}{d^n}-e^{i \theta_0}|_v)\\&=\log |\beta^\frac{1}{d^n}|_v + \log |e^\frac{2 \pi i}{d^n}-e^{i \theta_0}|_v.
	\end{align*}
By Theorem \ref{linearform}, 
	\begin{align*}
		\log|\gamma - \alpha|_v &\geq \log s - C_\epsilon[K:\mathbb{Q}]^3(h(\alpha)+1)|\mathcal{P}|^\epsilon
		\\& \geq -C_\epsilon [K:\mathbb{Q}]^3(h(\alpha)+h(s)+1)|\mathcal{P}|^\epsilon.
	\end{align*} Hence,
	$$ \max_{\gamma \in \mathcal{P}} \log |\gamma - \alpha|_v^{-1} < C_\epsilon[K:\mathbb{Q}]^3(h(\alpha)+h(s)+1)|\mathcal{P}|^\epsilon.$$
	This completes the proof of Proposition \ref{prop3.1}.
\end{proof}

\begin{proposition}\label{prop3.2}
	Let $K$ be a number field, $\varphi(z)=z^d$ be a rational map of degree $d\geq 2$ and $\beta$ be a fixed non-zero element in $K$. Fix a non-archimedean place $v$ of $\mathbb{Q}$ corresponding to the prime $p$ along with an extension to $\overline{\mathbb{Q}}$, let $D$ be a positive integer and $\delta >0$. Then there exists a constant $C_5 >0$ (depending $p, \delta, D$) such that for any $\alpha \in \mathbb{P}^1(\overline{K})$ with $[K:\mathbb{Q}] < D$ and $\gamma \in \mathcal{O}^-_\varphi(\beta)$,  we have 
	$$ \log|\gamma - \alpha |_v^{-1} < \delta$$
	if $|G_K(\gamma)| > C_5$. 
\end{proposition}
\begin{proof}
Let $\epsilon \in (0,1)$ be a real number. For $\sigma_1, \sigma_2 \in \text{Gal}(\overline{K}/K)$, suppose $\gamma_{k}:=\sigma_1(\gamma)$ and $\gamma_{l}:=\sigma_2(\gamma)$
with $\gamma_k, \gamma_l \in \mathcal{O}^-_\varphi(\beta)$. That is, $\gamma_k$ and $\gamma_l$ are in $\overline{K}_v$ satisfying $\gamma_k^{d^n}=\beta$ and $\gamma_l^{d^n}=\beta$ for some positive integer $n$. If we choose $r=1-\epsilon$, then by using arguments as in the proof of Lemma \ref{lemma1}, we get $|\gamma_k-\alpha|_v<1-\epsilon, |\gamma_l-\alpha|_v<1-\epsilon$. The only roots of unity $\zeta \in \overline{K}_v$ satisfying the condition $|1 - \zeta|_v \leq 1 - \epsilon$ are those of order $p^n$ for some integer $n \geq 0$. Moreover, if $\zeta$ has order divisible by at least two distinct prime numbers, then $\zeta - 1$ must be a unit.
	If $n \geq 1$, we have:
	\begin{equation*}
		1 - \epsilon \geq |\zeta - 1|_v = p^{-\frac{1}{(p - 1)p^{n - 1}}} \geq p^{-1/p^{n - 1}}.
	\end{equation*}
	Taking logarithms on both sides yields
	\begin{equation*}
		- \log(1 - \epsilon) \leq \frac{\log p}{p^{n - 1}}.
	\end{equation*}
	Assuming $- \log(1 - \epsilon) \geq \epsilon$, we obtain
	\begin{equation*}
		p^{n - 1} \leq \frac{\log p}{\epsilon}.
	\end{equation*}
	Let $n_0$ be the largest integer such that $p^{n_0} \leq p(\log p)\epsilon^{-1}$. Then clearly $n_0 \geq 0$ and $n \leq n_0$, implying that $\zeta^{p^{n_0}} = 1$. Thus, there are at most $p(\log p)\epsilon^{-1}$ possibilities for $\zeta$ satisfying $|1 - \zeta|_v \leq 1 - \epsilon$.
	Consequently, there are at most $C_5 := p(\log p)\epsilon^{-1}$ Galois conjugates of $\gamma$ such that
	\begin{equation*}
		|\gamma - \alpha|_v < p^{-\frac{1}{(p - 1)p^{n - 1}}}.
	\end{equation*}
	Therefore, if $|G_K(\gamma)| > C_5$, we must have
	\begin{equation*}
		|\gamma - \alpha|_v \geq p^{-\frac{1}{(p - 1)p^{n - 1}}},
	\end{equation*}
	and hence,
	\begin{equation*}
		\log |\gamma - \alpha|_v^{-1} < \frac{1}{p^{n - 1}(p - 1)} \log p < \delta.
	\end{equation*}
\end{proof}
From the proof of the above result, we have the following:
\begin{coro}\label{cor3.3}
	Let $K$ be a number field, let $\varphi(z)=z^d$ be a rational map of degree $d\geq 2$ and $\beta$ be a fixed non-zero element in $\mathbb{Q}$. Fix a non-archimedean place $v$ of $K$ corresponding to the prime $p$. Then for any $\alpha \in \mathbb{P}^{1}(\overline{K})$, there do not exist two distinct points $\gamma_1,\gamma_2$ in $\mathcal{O}_\varphi^-(\beta)$ such that $$\log|\gamma_i-\alpha|_v^{-1} \ge\frac{1}{p-1}\log p,
	$$ for $i=1, 2$.
\end{coro}
\begin{proof}
	Let $\alpha \in \mathbb{P}^{1}(\overline{K})$ be arbitrary. Suppose, on the contrary, that we have\begin{align*}
		&\log|\gamma_i-\alpha|_v^{-1} \ge \frac{1}{p-1}\log p = \log p^{\frac{1}{p-1}}
	\end{align*} for $i=1, 2$.
	This implies $|\gamma_i-\alpha|_v < 1/p^{\frac{1}{p-1}}$.  Thus, 
	$$|\gamma_1-\gamma_2|_v =|\gamma_2|_v|\zeta-1|_v <1/p^{\frac{1}{p-1}},$$ for $\zeta=\gamma_1 \gamma_2^{-1}$ which is a contradiction, as $|1-\zeta|_v$ is at least $1/p^{\frac{1}{p-1}}$ for any root of unity $\zeta$.
\end{proof}
\begin{proposition}\label{degree}
Let $\beta \in K $ be neither zero nor a root of unity and let $\gamma$ be a solution of $z^n=\beta$ for some integer $n \in \mathbb{Z}_{> 0}$. Then, for all sufficiently large integers $n$,
\begin{equation}
[K(\gamma) : K] \ge \dfrac{\sqrt{n}}{[K:\Q]}.
\end{equation}
\end{proposition}
\begin{proof}
Consider the polynomial $z^n-\beta \in K[z]$. Since $\beta$ is not a root of unity, the equation $z^n=\beta$ has exactly $n$ distinct roots in an algebraic closure $\overline K$ of $K$. Fix one root $\gamma$, and we write the full set of roots as $\{\gamma \zeta_n^i : i=0,1,\dots,n-1\}$, where $\zeta_n$ is a primitive $n$-th root of unity. For each $i$, let $\sigma_i$ be a $K$-embedding of $K(\gamma \zeta_n)$ into $\overline K$ such that $\sigma_i(\gamma)=\gamma\zeta_n^i$. Then
\begin{equation*}
\frac{\sigma_i(\gamma)}{\gamma}=\zeta_n^i.
\end{equation*}
For some $i$ with $\gcd(i,n)=1$, the element $\zeta_n^i$ is a primitive $n$-th root of unity, and hence $\zeta_n \in K(\gamma \zeta_n)$. Therefore $K(\zeta_n) \subseteq K(\gamma \zeta_n)$, which implies
\begin{equation*}
[K(\gamma): K]=[K(\gamma \zeta_n):K] \ge [K(\zeta_n):K] \ge [\mathbb Q(\zeta_n):\mathbb Q]/[K:\mathbb Q].
\end{equation*}
It is well known that
\begin{equation*}
[\mathbb Q(\zeta_n):\mathbb Q]=\varphi(n),
\end{equation*}
where $\varphi$ denotes Euler’s totient function. Moreover, for all sufficiently large $n$ we have $\varphi(n) \ge \sqrt n$ (see \cite[Theorem 327]{hardy}). Consequently,
\begin{equation*}
[K(\gamma):K] \ge \frac{\sqrt n}{[K:\mathbb Q]}.
\end{equation*}
This completes the proof.
\end{proof}
The following result gives a bound on the logarithmic equidistribution rate. 
\begin{proposition}\label{quntilogprop2}
    Let $K$ be a number field, $\varphi(z)=z^d$ be a rational map of degree $d\geq 2$ and $\beta$ be a fixed non-zero element in $K$, $\mathcal{P}$ be the Galois orbit of some point in $\mathcal{O}_\varphi^-(\beta)$. Let $v \in M_K$ be a place of $K$ that is extended to $\overline{K}$ and $\alpha \in \mathbb{P}^1(K)$ be a point. Let $s:=|\gamma|_v$ for some $\gamma \in \mathcal{O}_\varphi^-(\beta)$ and fix some $\delta$ with $0<\delta< \frac{1}{2}$. Then there exists a constant $C_7 = C_7([K:\mathbb{Q}], \delta) > 0$ such that for any $A > 1$, if 
	$$
	\max_{z \in \mathcal{P}} \log \left| z - \alpha \right|_v^{-1} < A[K:\mathbb{Q}]^3(h(\alpha) +h(s)+ 1) |\mathcal{P}|^{1/2 - \delta},
	$$
	then
	$$
	\left |\frac{1}{|\mathcal{P}|} \sum_{z \in \mathcal{P}} \lambda_{\alpha, v}(z) - \int \lambda_{\alpha, v}(z)  \, d\mu_{\varphi,v}\right |_v
	\leq  \frac{C_7}{|\mathcal{P}|^{\delta}} \sqrt{\log |\mathcal{P}|} A \left(h(\alpha)+ h(s)+ 1\right).
	$$
\end{proposition}
\begin{proof}
Let $\max_{z \in \mathcal{P}} \log \left| z - \alpha \right|_v^{-1} =C_6, [K:\mathbb{Q}]=D$ and let $M=1/\tau>1$ be a real number. Let $v$ be an archimedean place and if we take $|\mathcal{P}| \geq D^4$ with $M \geq |\mathcal{P}|^4/s$, then we claim that there is at most one $z$ inside $\mathcal{P}$ for which $\log \left| z - \alpha \right|_v^{-1} \geq \log M$.  
Consider the disk, $B(\alpha, M^{-1})=\{z\in \mathcal{P}:|z-\alpha|_v \leq M^{-1}\}$. If $z_1$ and $z_2$
are two distinct elements in the disk $B(\alpha, M^{-1})$, then 
\begin{equation}\label{prop4.5eq1}
    |z_1-z_2|_v \leq |z_1-\alpha|_v + |z_2-\alpha|_v\leq 2 M^{-1}.
\end{equation}
Also, 
\begin{align}\label{prop4.5eq2}
\begin{split}
    |z_1-z_2|_v &= |\beta^{1/d^n}|_v\left|e^{2\pi ik_1/n}-e^{2\pi ik_2/n}\right|_v \geq  s\cdot 2\left | \sin{\dfrac{\pi (k_1-k_2)}{n}}\right|_v \\
    &\geq 2s \frac{2}{\pi} \dfrac{\pi |k_1-k_2|_v}{n}\geq \dfrac{s}{n} \geq \dfrac{s}{D^2|\mathcal{P}|^2}.
    \end{split}
\end{align}
Note that we use Proposition \ref{degree} in the last inequality of \eqref{prop4.5eq2}. From \eqref{prop4.5eq1} and \eqref{prop4.5eq2}, we get $|\mathcal{P}|^2 \leq 2 D^2$, which is a contradiction because $|\mathcal{P}| \geq D^4$. So, there exists at most one $z$ inside $\mathcal{P}$ for which $|z-\alpha|_v \leq M^{-1}$, i.e., $\log |z-\alpha|_v^{-1}\geq \log M$.

We split $\mathcal{P}$ into two disjoint sets $\mathcal{P}_1$ and $\mathcal{P}_2$, where $\mathcal{P}_1=\{z\in \mathcal{P}:\log |z-\alpha|_v^{-1} \leq \log M\}$ and $\mathcal{P}_2=\mathcal{P} \setminus \mathcal{P}_1$. For archimedean $v$, we have $|\mathcal{P}_2| \leq 1$ from the above diskussion and for nonarchimedean $v$,  we also have $|\mathcal{P}_2| \leq 1$ by Corollary \ref{cor3.3}. 

Observe that $\lambda_{\tau, v}(z)$ in \eqref{truncation} can be written as $$\lambda_{\tau, v}(z) = \lambda_{M, v}(z) = \log^+|z|_v+\log^+|\alpha|_v+  \min(\log M, -\log |z-\alpha|_v).$$
Then for all $z \in \mathcal{P}_1$, we have $\lambda_{M, v}(z)=\lambda_{\alpha, v}(z)$ and hence
\begin{align}\label{E1}
    \left |\frac{1}{|\mathcal{P}|} \sum_{z \in \mathcal{P}} \lambda_{M, v}( z) - \frac{1}{|\mathcal{P}|}\sum_{z \in \mathcal{P}} \lambda_{\alpha, v}(z)  \right |_v&=\left |\frac{1}{|\mathcal{P}|} \sum_{z \in \mathcal{P}_2} \lambda_{M, v}(z) - \frac{1}{|\mathcal{P}|}\sum_{z \in \mathcal{P}_2} \lambda_{\alpha, v}(z)  \right |_v \nonumber \\ &\leq \frac{1}{|\mathcal{P}|} \sum_{z \in \mathcal{P}_2} \left| \lambda_{M, v}(z) -\lambda_{\alpha, v}(z)\right|_v  \leq \frac{C_6}{|\mathcal{P}|}|\mathcal{P}_2|.
\end{align}
As $\mu_{\varphi, v}(B(z, \epsilon))=O(\epsilon)$, then we bound
\begin{align}\label{E2}
    \left|\int \lambda_{M, v}(z)  \, d\mu_{\varphi,v} - \int \lambda_{\alpha, v}(z)  \, d\mu_{\varphi,v}\right |_v &\leq \int \left |\lambda_{M, v}(z) -\lambda_{\alpha, v}(z) \right|_vd\mu_{\varphi,v} \nonumber \\
    & \leq \int_{|z-\alpha|_v \leq 1/M} \log |z-\alpha|_v^{-1} d\mu_{\varphi,v} \nonumber \\
    & \leq O\left(\dfrac{\log M}{M}\right).
\end{align}
Now using Proposition \ref{corollary2.2} for the map $\varphi(z)=z^d$ with $d \geq 2$, we have
\begin{equation}\label{E3}
     \left| \frac{1}{|\mathcal{P}|} \sum_{z \in \mathcal{P}} \lambda_{M,v}(z) - \int \lambda_{M,v}(z) d\mu_{\varphi,v} \right|_v \leq C_3   \left( \frac{1+\log|\mathcal{P}|^{1/2}}{|\mathcal{P}|^{1/2}}  \right).
\end{equation}
Now combining \eqref{E1}, \eqref{E2} and \eqref{E3} and choosing $\kappa <1/4$ with $M=|\mathcal{P}|^{1/{2\kappa}}$, we will get
\begin{align}
\begin{split}
    &\left |\frac{1}{|\mathcal{P}|} \sum_{z \in \mathcal{P}} \lambda_{\alpha, v}(z) - \int \lambda_{\alpha, v}(z)  \, d\mu_{\varphi,v}\right |_v 
   \leq \left |\frac{1}{|\mathcal{P}|} \sum_{z \in \mathcal{P}} \lambda_{M, v}( z) - \frac{1}{|\mathcal{P}|}\sum_{z \in \mathcal{P}} \lambda_{\alpha, v}(z)  \right |_v +\\ &  \left|\int \lambda_{M, v}(z)  \, d\mu_{\varphi,v} - \int \lambda_{\alpha, v}(z)\, d\mu_{\varphi,v}\right |_v +\left| \frac{1}{|\mathcal{P}|} \sum_{z \in \mathcal{P}} \lambda_{M,v}(z) - \int \lambda_{M,v}(z) d\mu_{\varphi,v} \right|_v  \\
    &\leq \frac{C_6}{|\mathcal{P}|}|\mathcal{P}_2|+O\left(\dfrac{\sqrt{\log |\mathcal{P}|}}{\sqrt{|\mathcal{P}|}}\right)+C_3   \left( \frac{1+\log|\mathcal{P}|^{1/2}}{|\mathcal{P}|^{1/2}}  \right)   \\
    &\leq \frac{C_6}{|\mathcal{P}|}+O\left(\dfrac{\sqrt{\log |\mathcal{P}|}}{\sqrt{|\mathcal{P}|}}\right)\\
    &\leq \frac{A[K:\mathbb{Q}]^3(h(\alpha) +h(s)+ 1) |\mathcal{P}|^{1/2 - \delta}}{|\mathcal{P}|}+O\left(\dfrac{\sqrt{\log |\mathcal{P}|}}{\sqrt{|\mathcal{P}|}}\right).
    \end{split}
\end{align}
In the third inequality, first term is bounded since $|\mathcal{P}_2| < 1$ and the sum of the other two terms is $O\left(\frac{\sqrt{\log |\mathcal{P}|}}{\sqrt{|\mathcal{P}|}}\right)$. 
 In the last inequality, we are using the hypothesis $C_6=\max_{z \in \mathcal{P}} \log \left| z - \alpha \right|_v^{-1} < A[K:\mathbb{Q}]^3(h(\alpha) +h(s)+ 1) |\mathcal{P}|^{1/2 - \delta}$.  Using a suitable constant $C_7=C_7([K:\mathbb{Q}], \delta)$ we will get the desired inequality.
\end{proof}
\subsection{Proof of Theorem \ref{thm2}}
Suppose that $D$ is a positive integer and we consider a finite extension $K$ of $\mathbb{Q}$ of degree at most $D$. Let $S'$ be the set of places in $M_K$ above $S$. Here we can observe that $|S'| \leq [K:\mathbb{Q}]|S| \leq D|S|$. 
For non-archimedean $v \in M_K$ we have $\log|z - \alpha|^{-1}_v<\delta$ by Proposition \ref{prop3.2}  and for archimedean $v\in M_K$, by Proposition \ref{prop3.1}, $$\max_{z \in \mathcal{P}} \log |z - \alpha|_v^{-1}<C_\epsilon[K:\mathbb{Q}]^3(h(\alpha)+h(s)+1)|\mathcal{P}|^\epsilon.$$ Then from Proposition \ref{quntilogprop2}, there is a constant $C_7>0$, depending only on $\varphi$, such that for any $v\in S'$, we have    
\begin{equation}\label{eq4.1}
	\left|\frac{1}{|\mathcal{P}|}\sum_{z \in \mathcal{P}} \lambda_{\alpha, v}(z)-\int \lambda_{\alpha, v}(z) d\mu_{\varphi, v}\right|_v \le \frac{C_7}{|\mathcal{P}|^\delta} \sqrt{\log |\mathcal{P}|}A(h(\alpha)+ h(s)+1).
\end{equation}
Suppose that $|\mathcal{P}|>C:=C(D,S)$ is large enough. By taking the summation of $v \in S'$ in \eqref{eq4.1}, we obtain
\begin{equation}\label{bound}
    \sum_{v \in S'}\left|\frac{1}{|\mathcal{P}|}\sum_{z \in \mathcal{P}}N_v\lambda_{\alpha, v}(z)- \int N_v\lambda_{\alpha, v}(z) d\mu_{\varphi, v}\right|_v \leq \frac{h(\alpha)+h(s) +1}{D^5}.
\end{equation}
Using Proposition \ref{propadelic}, we will get 
\begin{align*}
	\big|h_{\mathcal{L}}(\mathcal{P})-\big<{\mathcal{L}},\mathcal{L}_\varphi\big>\big|_v &\le C_{AZ, d} \Big( \frac{1+\log |\mathcal{P}|^{1/2}}{|\mathcal{P}|^{1/2}}\Big) < \frac{1}{D^{1.5}}.
\end{align*}
As $\gamma$ is $S'$-integral relative to $\alpha$, we have $\lambda_{\alpha, v}(z)=0$ for all $v \notin S'$ and $\gamma \in \mathcal{P}$.
Therefore,
\begin{align}\label{al4.10}
	\frac{1}{|\mathcal{P}|}\sum_{v \in M_{K}}\sum_{z \in \mathcal{P}}N_v\lambda_{\alpha, v}(z)-&\sum_{v \in M_{K}} \int N_v\lambda_{\alpha, v}(z) d\mu_{\varphi, v} \nonumber \\  &=h_{L}(\mathcal{P})-\big<{L},L_\varphi\big> +h_\varphi(\alpha) \nonumber \\ 
	& \ge -\frac{1}{D^{1.5}}+h_{\varphi}(\alpha).
\end{align}
Putting \eqref{al4.10} in \eqref{bound} and then simplifying, we get
\begin{equation}
    h_\varphi(\alpha)-\frac{h(\alpha)}{D^5}\leq \frac{1}{D^{1.5}}+\frac{h(s) +1}{D^5}.
\end{equation}  
Since $|h_\varphi(\alpha)-h(\alpha)| < C_\varphi$, where $C_\varphi$ is an absolute constant, then we have \begin{equation*}
    h(\alpha)\left(1-\frac{1}{D^5}\right)\leq C_\varphi +\frac{1}{D^{1.5}}+\frac{h(s) +1}{D^5}.
\end{equation*}
By a result of Dobrolowski \cite{dob}, we have
$h(\alpha) \geq O\!\left(\tfrac{1}{D^{3/2}}\right)$. Consequently, for some constant $C_8>0$, we have
\begin{align*}
   \frac{C_8}{D^{1.5}}\left(1-\frac{1}{D^{5}}\right)
   &\leq C_{\varphi}+\frac{1}{D^{1.5}}+\frac{h(s) +1}{D^5}.
\end{align*}
Then after simplifications, we deduce that
\begin{equation*}
D^{1.5} \geq  C_9(D^{5}-1),
\end{equation*}
which is impossible for $D>1$.
Hence, $|\mathcal{P}| < C$. This completes the proof of Theorem \ref{thm2}. \qed

\subsection{Proof of Theorem \ref{thm02}}
We begin the proof by applying Proposition \ref{prop3.1} and Corollary \ref{cor3.3}. In particular, there exists a constant $A_{\epsilon}$ such that, for every place $v \in S$, where $S$ is the set of places of $K$, 
\begin{equation*}
	\max_{z \in \mathcal{P}} \log |z - \alpha|_v^{-1} < A_\epsilon D^3(h(\alpha)+h(s)+1)|\mathcal{P}|^\epsilon
\end{equation*}
for all $\mbox{Gal}(\overline{K}/K)$-orbits $\mathcal{P}$ of points in $\mathcal{O}_\varphi^-(\beta)$ with the possible exception of one orbit for each place $v$. Thus, we get a total $|S_{\mbox{fin}}|$ exceptions.

Now, let $\gamma \in \mathcal{O}^-_\varphi(\beta)$ and is not one of the  above exceptions. Then by Proposition \ref{quntilogprop2}, there is a constant $C_7>0$, depending only on $d$, such that for any $v\in S$, we have
\begin{equation}\label{eq3.2}
	\left|\frac{1}{|\mathcal{P}|}\sum_{z \in \mathcal{P}}\lambda_{\alpha, v}(z)-\int\lambda_{\alpha, v}(z) d\mu_{\varphi, v}\right|_v \le \frac{C_7}{|\mathcal{P}|^\delta} \sqrt{\log |\mathcal{P}|}A(h(\alpha)+ h(s)+1).
\end{equation}  
Setting $\delta=\frac{1}{2}-\epsilon, A=A_\epsilon (\log D)^2D^3$ in \eqref{eq3.2}, we get
\begin{align*}
	\Big|\frac{1}{|\mathcal{P}|}\sum_{z \in \mathcal{P}} &\lambda_{\alpha, v}(z)-\int\lambda_{\alpha, v}(z) d\mu_{\varphi, v}\Big|_v \\
	&\le \frac{CA_\epsilon}{|\mathcal{P}|^{\frac{1}{2}-\epsilon}} \sqrt{\log |\mathcal{P}|}D^3(\log D)^2(h(\alpha)+ h(s)+1).
\end{align*}
For any constant $N > 0$ assuming that $|\mathcal{P}|>C_{10}|S|^3D^{8}$ for some suitable $C_{10}$ gives, 
\begin{equation*}
	\left|\frac{1}{|\mathcal{P}|}\sum_{z \in \mathcal{P}}\lambda_{\alpha, v}(z)-\int\lambda_{\alpha, v}(z) d\mu_{\varphi, v}\right|_v \le \frac{h(\alpha)+ h(s)+ 1}{N|S|D^{1.5}}.
\end{equation*}
Summing up over all places in $S$, we get
\begin{align*}
	\sum_{v \in S}\Big|\frac{1}{|\mathcal{P}|}\sum_{z \in \mathcal{P}}N_v\lambda_{\alpha, v}(z)&- \int N_v\lambda_{\alpha, v}(z) d\mu_{\varphi, v}\Big|_v \\
	&\le \sum_{v \in S} N_v \frac{h(\alpha)+ h(s)+ 1}{N|S|D^{1.5}} \\
	& = \sum_{v \in S} N_v \frac{h(\alpha)+1}{N|S|D^{1.5}} + \sum_{v \in S} N_v \frac{h(s)}{N|S|D^{1.5}} \\
	& \le \frac{h(\alpha)+1}{ND^{1.5}}+ \sum_{v \in S} N_v \frac{h(s)}{N|S|D^{1.5}} 
	\\&\le \frac{h(\alpha) +1}{ND^{1.5}}+ \frac{h(s)}{N|S|D^{1.5}}\\&\le \frac{h(\alpha) +1}{ND^{1.5}}.
\end{align*}
Again by a similar argument used in the proof of Theorem \ref{thm2} leads to the following conclusion 
\begin{equation*}
 h(\alpha)\le C_\varphi+\frac{h(\alpha)+1}{ND^{1.5}} + \frac{1}{D^{1.5}}.
\end{equation*}
If $h(\alpha) \ge 1$, then $ND^{1.5}\leq \dfrac{N+2}{1-C_\varphi}$, which is not possible for $D>1$. Now if  $h(\alpha) <1$, then using Dobrowolski's result \cite{dob} implies that 
$$h(\alpha) \ge \frac{C_{11}}{D(\log D)^3}.$$
This gives,
\begin{align*}
	\frac{C_{11}}{D(\log D)^3}\Big(1-\frac{1}{ND^{1.5}}\Big)
	&\le C_\varphi+\frac{1}{ND^{1.5}}+\frac{1}{D^{1.5}}.
\end{align*}
If $N$ is sufficiently large, then 
\begin{equation}
	\frac{C_{11}}{D(\log D)^3} \le C_\varphi+\frac{1}{D^{1.5}}.
\end{equation}
However, this leads to a contradiction. Hence, $|\mathcal{P}| \le C_{10}|S|^3 D^{8}$. This completes the proof of Theorem \ref{thm02}. \qed

\textbf{Acknowledgments:}
The Authors are supported by a grant from National Board for Higher Mathematics (NBHM), Sanction Order No: 14053. Furthermore, S.S.R. is partially supported by grant from Science and Engineering Research Board (SERB)(File  No.:CRG/2022/000268).

\end{document}